\UseRawInputEncoding 
\documentclass{amsart}

\usepackage[title]{appendix}
\usepackage{amssymb}
\usepackage{amsmath}
\usepackage{amsthm}
\usepackage{tikz-cd}
\usepackage{mathtools}

\usepackage{framed}
\usepackage{comment}

\usepackage{hyperref}
\usepackage[capitalize]{cleveref}

\newcommand{\GL}{\text{GL}}

\newcommand{\Q}{\mathbb{Q}}
\newcommand{\R}{\mathbb{R}}

\newcommand{\Z}{\mathbb{Z}}
\renewcommand{\H}{\mathbb{H}}

\newcommand{\PP}{\mathcal{P}}

\renewcommand{\l}{\Lambda}

\DeclareMathOperator{\Aut}{Aut}

\DeclareMathOperator{\Hom}{Hom}
\DeclareMathOperator{\id}{id}

\DeclareMathOperator{\rk}{rk}

\DeclareMathOperator{\rank}{rank}
\DeclareMathOperator{\End}{End}

\DeclareMathOperator{\Res}{Res}

\DeclareMathOperator{\Ext}{Ext}
\DeclareMathOperator{\LF}{LF}
\DeclareMathOperator{\Def}{def}
\DeclareMathOperator{\IM}{Im}

\DeclareMathOperator{\coker}{coker}

\DeclareMathOperator{\Syl}{Syl}

\DeclareMathOperator{\Alg}{Alg}
\DeclareMathOperator{\D}{D2}
\DeclareMathOperator{\PHT}{PHT}
\DeclareMathOperator{\Proj}{hProj}

\newtheorem{thm}{Theorem}[section]

\newtheorem*{thm*}{Theorem}
\newtheorem{prop}[thm]{Proposition}
\newtheorem*{prop*}{Proposition}
\newtheorem{lemma}[thm]{Lemma}
\newtheorem{cor}[thm]{Corollary}
\newtheorem*{cor*}{Corollary}
\newtheorem{conj}[thm]{Conjecture}

\theoremstyle{definition}

\newtheorem{question}[thm]{Question}
\newtheorem*{question*}{Question}

\newtheorem*{theorem*}{Theorem}
\theoremstyle{remark}
\newtheorem{remark}[thm]{Remark}
\newtheorem*{remark*}{Remark}

\newtheoremstyle{named}{}{}{\itshape}{}{\bfseries}{.}{.5em}{\thmnote{#3's }#1}
\theoremstyle{named}
\newtheorem*{d2}{D2 problem}

\newtheoremstyle{custom}{}{}{\itshape}{}{\bfseries}{.}{.5em}{\thmnote{#3}#1}
\theoremstyle{custom}
\newtheorem*{cproblem}{}

\usepackage{rotating}

\usepackage{color}
\usepackage{enumerate}

\newcommand{\wh}{\widehat}
\newcommand{\wt}{\widetilde}

\makeatletter
\@namedef{subjclassname@2020}{%
  \textup{2020} Mathematics Subject Classification}
\makeatother

\begin{document}

\title{On CW-complexes over groups with periodic cohomology}

\author{John Nicholson}
\address{Department of Mathematics, UCL, Gower Street, London, WC1E 6BT, U.K.}
\email{j.k.nicholson@ucl.ac.uk}



\begin{abstract}
If $G$ has $4$-periodic cohomology, then finite D2 complexes over $G$ are determined up to polarised homotopy by their Euler characteristic if and only if $G$ has at most two one-dimensional quaternionic representations. We use this to solve Wall's D2 problem for several infinite families of non-abelian groups and, in these cases, also show that any finite Poincar\'{e} $3$-complex $X$ with $G=\pi_1(X)$ admits a cell structure with a single $3$-cell. 
The proof involves cancellation theorems for $\Z G$ modules where $G$ has periodic cohomology.
\vspace{-5mm}
\end{abstract}

\maketitle

\section{Introduction}

In the 1960s, C. T. C. Wall defined a D$n$ complex to be a CW-complex $K$ that $H_i(\wt K)=0$ for $i > n$ and $H^{n+1}(K; M)=0$ for all finitely generated $\Z [\pi_1(K)]$-modules $M$. 
Wall showed that a finite D$n$ complex is homotopy equivalent to a finite $n$-complex provided $n > 2$ \cite{Wa65}. This was extended to the case $n=1$ by Stallings-Swan \cite{St68, Sw69}, but the question for $n=2$ remains completely open:

\begin{d2}
Is every finite {\normalfont D2} complex homotopy equivalent to a finite $2$-complex?
\end{d2}

We say that a finitely presented group $G$ has the \textit{$D2$ property} if every finite D2 complex $X$ with $\pi_1(X)=G$ is homotopy equivalent to a finite $2$-complex. 

In \cite{Wa67}, Wall applied this to the question of whether a finite Poincar\'{e} $n$-complex $X$ could be given a cell structure with a single $n$-cell $e^n$. He showed that
\[ X = K \cup e^n \]
where $K$ is a finite D$(n-1)$-complex and so, if $K$ is homotopy equivalent to a finite $(n-1)$-complex, then $X=K \cup e^n$ gives a cell structure with a single $n$-cell.
In particular, if $n \ne 3$, then every finite Poincar\'{e} $n$-complex admits a cell structure with a single $n$-cell. The case $n=3$ remains open though would be true provided the fundamental groups of finite Poincar\'{e} $3$-complexes had the D2 property.

Since the finite fundamental groups of Poincar\'{e} $3$-complexes have $4$-periodic cohomology, the D2 property for such groups is of special interest.
In fact, it was conjectured by J. M. Cohen that all other groups have the D2 property \cite[p415]{Co77}. 
In light of this, groups with 4-periodic cohomology have been the source of many proposed counterexamples \cite{Co77, Jo04, BW05}.
The only cases previously known to have the D2 property are cyclic, dihedral, or binary polyhedral \cite{Dy75, Jo03a}.

Our main result is the following partial classification of finite D2 complexes whose fundamental group has 4-periodic cohomology.
Let $m_{\H}(G)$ denote the number of copies of $\H$ in the Wedderburn decomposition of $\R G$ for a finite group $G$. 

\begingroup
\setcounter{thm}{0}
\renewcommand\thethm{\Alph{thm}}
\begin{thm} \label{thm:D2}
Let $G$ have $4$-periodic cohomology. Then finite {\normalfont D2} complexes $X$ with $\pi_1(X) = G$ are determined up to (polarised) homotopy by their Euler characteristic if and only if $m_{\H}(G) \le 2$.
\end{thm}

Here we assume that each finite D2 complex $X$ comes equipped with an isomorphism $\rho_X : \pi_1(X) \cong G$ and, for pairs $(X,\rho_X)$ and $(Y,\rho_Y)$, we say that a homotopy equivalence $h: X \to Y$ is a polarised homotopy equivalence if $\rho_Y \circ \pi_1(h) \circ \rho_X^{-1} = \id_G$. 

Recall that every finite $2$-complex $X$ with $\pi_1(X) = G$ is homotopy equivalent to the Cayley complex $X_{\mathcal{P}}$ of some presentation $\mathcal{P} = \langle s_1, \cdots, s_n \mid r_1, \cdots, r_m \rangle$ of $G$, and that $\chi(X_{\mathcal{P}}) = 1 - \Def(\mathcal{P})$ where $\Def(\mathcal{P}) = n-m$ is the \textit{deficiency} of $\mathcal{P}$. If $G$ is finite, then $\Def(\mathcal{P}) \ge 0$ (see, for example, \cite[Corollary 1.3]{Sw65}). In \cref{thm:main1} we show that, for every group $G$ with 4-periodic cohomology, there exists a finite D2 complex $X$ for which $\pi_1(X) = G$ and $\chi(X) = 1$. In particular, we will show:

\renewcommand\thethm{\Alph{thm}}
\begin{cor} \label{thm:D2-2}
Let $G$ have $4$-periodic cohomology. Then:
\begin{enumerate}[\normalfont(i)]
\item If $G$ has the {\normalfont D2} property, then $G$ has a balanced presentation.
\item If $G$ has a balanced presentation and $m_{\H}(G) \le 2$, then $G$ has the {\normalfont D2} property.
\end{enumerate}
\end{cor}

This builds upon the work of F. E. A. Johnson \cite{Jo03a} who considered the special case where there exists a 4-periodic free resolution of $\Z G$ modules. In particular, for such groups $G$, he proved (i) and also obtained a version of (ii) subject to stably free $\Z G$ modules being free (which we resolve in \cref{thm:main5}).
 
This has the following consequence for the question of whether a finite Poincar\'{e} $3$-complex has a cell structure with a single $3$-cell: 
 
\renewcommand\thethm{\Alph{thm}}
\begin{cor} \label{thm:poincare}
Let $X$ be a finite Poincar\'{e} $3$-complex with $G=\pi_1(X)$ finite. Then:
\begin{enumerate}[\normalfont(i)]
\item If $X$ has a cell structure with a single $3$-cell, \hspace{-0.5mm}then $G$ \hspace{-0.5mm}has a balanced presentation.
\item If $G$ has a balanced presentation and $m_{\H}(G) \le 2$, then $X$ has a cell structure with a single $3$-cell.
\end{enumerate}
\end{cor}
\endgroup
\setcounter{thm}{0} 
 
In \cref{thm:4-periodic-list}, we show that the groups $G$ with $4$-periodic cohomology for which $m_{\H}(G) \le 2$ are as follows. Here we use the notation of Milnor \cite{Mi57}, and each family contains $G \times C_n$ for any $G$ listed and any $n \ge 1$ coprime to $|G|$.
\begin{enumerate}[\normalfont(i)]
\item $C_n$ for $n \ge 1$, the cyclic groups of order $n$.
\item $D_{4n+2}$ for $n \ge 1$, the dihedral groups of order $4n+2$.
\item $Q_8, Q_{12},Q_{16},Q_{20},\widetilde{T},\widetilde{O},\widetilde{I}$.
\item $D(2^n,3)$, $D(2^n,5)$ for $n \ge 3$.
\item $P_{8 \cdot 3^n}'$ for $n \ge 2$.
\item $P_{48 n}''$ for $n \ge 3$ odd.
\item $Q(16;m,n)$ for $m > n \ge 1$ odd coprime.
\end{enumerate}

By considering which of these groups have balanced presentations, we will show the following. This was previously shown by M. N. Dyer \cite{Dy75} for the groups in (i) and by Johnson \cite{Jo03a} for the groups in (ii) and for many of the groups in (iii).

\begingroup
\renewcommand\thethm{\ref{thm:D2-property}}
\begin{thm}
Suppose $G$ is in $\text{\normalfont(i)}$-$\text{\normalfont(v)}$ or has the form $Q(16;n,1) \times C_k$ for some $n, k \ge 1$ odd coprime. Then $G$ has the {\normalfont D2} property.
\end{thm}

Finally, we consider the case $m_{\H}(G) \ge 3$. Let $Q_{4n}$ denote the quaternion group of order $4n$, which has 4-periodic cohomology and $m_{\H}(Q_{4n}) = \lfloor n/2 \rfloor$. By combining recent results of W. H. Mannan and T. Popiel \cite{MP21} with Theorem \ref{thm:main0}, we show: 

\renewcommand\thethm{\ref{thm:q28}}
\begin{thm}
$Q_{28}$ has the {\normalfont D2} property and $m_{\H}(Q_{28})=3$.
\end{thm}
\endgroup

This group was proposed as a counterexample in \cite{BW05} (see also \cite[p23]{MR18}). We also point out that the example of Mannan-Popiel gives a counterexample to one part of the conjecture of J. M. Cohen mentioned previously \cite[p381]{Wa79a}. The possibility remains that some group with 4-periodic cohomology does not have a balanced presentation and so would be a counterexample to the D2 problem.

We now proceed to outline the series of results which will lead to the proof of \cref{thm:D2}. 
\textit{From now on, we will take $G$ to be a finitely presented group and we will assume all $\Z G$ modules are finitely generated left modules.}

\subsection{Reduction to algebra} \label{subsection:1.1}
Let $\D(G)$ denote the set of polarised homotopy types of finite D$2$ complexes $X$ with $\pi_1(X) = G$. This is a graded graph with grading given by the Euler characteristic $\chi(X)$ and edges between between each $(X,\rho)$ and $(X \vee S^2, \rho^+)$ where $\rho^+$ is induced by the collapse map $X \vee S^2 \to X$. 

Let $\Alg(G,2)$ denote the set of chain homotopy classes of algebraic $2$-complexes over $\Z G$, which are chain complexes $(F_*,\partial_*)$ of the form
\[
F_2 \xrightarrow[]{\partial_2} F_1 \xrightarrow[]{\partial_1} F_0
\]
where the $F_i$ are free and $H_0(F_*) \cong \Z$ where $\Z$ has trivial $G$-action.
 This is a graded graph with grading $\chi(E)= \rank(F_2) - \rank(F_1)+\rank(F_0)$ and edges between each $E=(F_*, \partial_*)$ and the stabilised complex $E \oplus \Z G = (F_2 \oplus \Z G \xrightarrow[]{(\partial_2,0)} F_1 \xrightarrow[]{\partial_1} F_0)$.

In Section \ref{section:PHT}, we will establish the following:

\begingroup
\renewcommand\thethm{\ref{thm:main0}}
\begin{thm} 
Let $G$ be a finitely presented group. Then there exists an isomorphism of graded trees
\[ \widetilde{C}_* : \D(G) \to \Alg(G,2)\]
which is the same as the cellular chain map $X \mapsto C_*(\widetilde{X})$ when $X$ is a $2$-complex.
\end{thm}
\endgroup

This is presumably well known (see, for example, \cite[Theorem 4]{Wa66}), though we could not locate an equivalent statement in the literature. 
The proof we give is an adaption of the proof of the Realisation Theorem of Johnson \cite{Jo03b}.

Recall that, if $P$ is a projective $\Z G$ module, then $P \otimes \Q \cong \Q G^r$ is a free $\Q G$ module \cite{Sw60a}, and we define the rank of $P$ to be $\rank(P) = r$.
Let $C(\Z G)$ denote the projective class group and observe that a class $[P] \in C(\Z G)$ can also be viewed as the set of (non-zero) projective $\Z G$ modules $P_0$ for which $P \oplus \Z G^i \cong P_0 \oplus \Z G^j$ for some $i, j \ge 0$. 
This has the structure of a graded graph with grading given by $\rank(P_0)$ and edges between each $P_0$ and $P_0 \oplus \Z G$. 

For a group $G$ with $n$-periodic cohomology, let $\sigma_n(G) \in C(\Z G)/T_G$ denote the \textit{Swan finiteness obstruction} where $T_G \le C(\Z G)$ denotes the Swan subgroup (see \cref{section:preliminaries}).
By adapting an approach of Johnson \cite[Theorem 62.1]{Jo03a}, we show:

\begingroup
\renewcommand\thethm{\ref{thm:main1}}
\begin{thm}
Let $G$ have $4$-periodic cohomology. Then there is an isomorphism of graded trees
\[ \Psi: \Alg(G,2) \to [P_G]\]
for any projective $\Z G$ module $P_G$ for which $\sigma_4(G) = [P_G] \in C(\Z G)/T_G$.
\end{thm}
\endgroup

An explicit description for the map $\Psi$ can be found in \cref{prop:Psi-explicit}.

\subsection{Cancellation for projective modules} \label{subsection:1.2}

If $P$ is a projective $\Z G$ module, then we say that \textit{$[P]$ has cancellation} if $P_1 \oplus \Z G \cong P_2 \oplus \Z G$ implies $P_1 \cong P_2$ for all $P_1, P_2 \in [P]$. It was shown by H. Jacobinski \cite[Theorem 4.1]{Ja68} that $[P]$ has cancellation provided $G$ is finite and satisfies the \textit{Eichler condition}, i.e. $m_{\H}(G)=0$. 
The case $m_{\H}(G) \ne 0$ was studied extensively by Swan \cite{Sw62, Sw83}. 
However, in spite of the wide applicability of these results, 
the only groups with 4-periodic cohomology to which they apply are the groups in (i) and (ii). 

The main contribution of this article is the following general cancellation theorem for projective $\Z G$ modules when $G$ is a finite group. Let $\LF_1(\Z G)$ denote the set of projective $\Z G$ modules of rank one and recall from algebraic K-theory that $K_1(\Z G)=\GL(\Z G)^{ab}$ where $\GL(\Z G) = \bigcup_n \GL_n(\Z G)$.

\begingroup
\renewcommand\thethm{\ref{thm:main2}}
\begin{thm}
Let $G$ be a finite group with quotient $H=G/N$, let $\bar{P} \in \LF_1(\Z G)$ and let $P = \bar{P} \otimes_{\Z N} \Z \in \LF_1(\Z H)$. Suppose $m_{\H}(G)=m_{\H}(H)$ and that the map 
\[ \Aut (P) \to K_1(\Z H)\]
is surjective. Then $[\bar{P}]$ has cancellation if and only if $[P]$ has cancellation.
\end{thm}
\endgroup

This generalises a result of the author \cite[Theorem A]{Ni21} which dealt with the case of stably free $\Z G$ modules, i.e. where $\bar{P} = \Z G$ and $P = \Z H$.
 
Our first application generalises the main result in \cite{Sw83}. We say that $\Z G$ has \textit{stably free cancellation} (SFC) if every stably free $\Z G$ module is free or, equivalently, if the class $[\Z G]$ has cancellation.

\begingroup
\renewcommand\thethm{\ref{thm:main4}}
\begin{thm}
Let $G$ have periodic cohomology. 
Then $\Z G$ has {\normalfont SFC} if and only if $m_{\H}(G) \le 2$.
\end{thm}
\endgroup

This completely determines the groups $G$ with periodic cohomology for which $\Z G$ has SFC and also corrects a mistake in \cite[p249]{Jo03a} where it was suggested that the groups in (vii) did not have SFC. Our second application is the following:

\begingroup
\renewcommand\thethm{\ref{thm:main5}}
\begin{thm} \label{thm:main5}
Let $G$ have $4$-periodic cohomology and let $P_G$ be a projective $\Z G$ module for which $\sigma_4(G) = [P_G] \in C(\Z G)/T_G$. 
Then $[P_G]$ has cancellation if and only if $m_{\H}(G) \le 2$.
\end{thm}
\endgroup

By Theorems \ref{thm:main0} and \ref{thm:main1}, $\D(G)$ is isomorphic to $[P_G]$ as a graded tree. Hence \cref{thm:main5} implies that $\D(G)$ has cancellation if and only if $m_{\H}(G) \le 2$. This completes the proof of \cref{thm:D2}. For a more detailed argument, see \cref{section:cancellation-periodic-groups}.

\section{Polarised homotopy types and algebraic 2-complexes} \label{section:PHT}

The aim of this section will be to prove the following. Our proof will be a mild generalisation of the arguments of Johnson in \cite{Jo03b}.

\begingroup
\begin{thm} \label{thm:main0}
Let $G$ be a finitely presented group. Then there exists an isomorphism of graded trees
\[ \widetilde{C}_* : \D(G) \to \Alg(G,2)\]
which is the same as the cellular chain map $X \mapsto C_*(\widetilde{X})$ when $X$ is a $2$-complex.
\end{thm}
\endgroup
 
We begin by noting that every finite D2 complex is a finite D3 complex and so is homotopy equivalent to a finite 3-complex by \cite[Theorem E]{Wa65}. We therefore lose no generality in assuming throughout that every finite D2 complex is a finite 3-complex. 
Let $(X,\rho) \in \D(G)$ and consider the cellular chain complex
\[
\begin{tikzcd}
	C_*(X) = (C_3(\widetilde{X}) \ar[r,"\partial_3"] & C_2(\widetilde{X}) \ar[r,"\partial_2"] & C_1(\widetilde{X}) \ar[r,"\partial_1"] & C_0(\widetilde{X}))
\end{tikzcd}
\]
where the $C_i(\wt X)$ are free $\Z[\pi_1(X)]$ modules under the monodromy action. We can use $\rho$ to identify this with a chain complex of $\Z G$ modules which we denote $C_*(X,\rho)$. We will now show the following which is also our definition for $\wt C_*$:

\begin{prop} \label{prop:C_*-well-def} Let $(X,\rho) \in \D(G)$. Then $C_*(X,\rho)$ is chain homotopy equivalent to an algebraic $2$-complex $E$ over $\Z G$. In particular, we define $\wt C_*(X,\rho) = E$.
\end{prop}

In order to prove this, we will need the following two lemmas. Note that, since $\IM(\partial_{3}) \subseteq \ker(\partial_2)$, there is a well-defined map $\wt \partial_2 : C_2(\wt X)/\IM(\partial_{3}) \to C_{1}(\wt X)$.

\begin{lemma}[\text{\cite[Proposition 6.6]{Jo03b}}] \label{lemma:chain-homotopy-equiv}
Let $(X,\rho) \in \D(G)$. Then there is a chain homotopy equivalence $\varphi: C_*(X,\rho) \to C'_*(X,\rho)$ where
\[
\begin{tikzcd}
	C'_*(X,\rho) = ( C_2(\widetilde{X})/\IM(\partial_{3}) \ar[r,"\widetilde{\partial}_2"] & 
C_{1}(\widetilde{X}) \ar[r,"\partial_{1}"] &
C_0(\widetilde{X})).
\end{tikzcd}
\] 
\end{lemma}

\begin{lemma}[\text{\cite[Proposition 6.5]{Jo03b}}] \label{lemma:sf} Let $(X,\rho) \in \D(G)$ and let $C_*(X,\rho) = (C_*(\widetilde{X}),\partial_*)_{0 \le * \le 3}$. Then $C_2(\widetilde{X})/\IM(\partial_3)$ is a stably free $\Z G$ module.
\end{lemma}

\begin{proof}[Proof of \cref{prop:C_*-well-def}]
By \cref{lemma:sf}, there exists $i, j \ge 0$ for which there is an isomorphism $f: C_2(\widetilde{X})/\IM(\partial_{3}) \oplus \Z G^i \cong \Z G^j$. We can now define a chain homotopy equivalence
\[
\begin{tikzcd}[row sep=.5cm, column sep=small]
C'_*(X,\rho) \ar[d,"f_*"] \\
E
\end{tikzcd}
= \left( 
\begin{tikzcd}[row sep=.5cm, column sep=small]
C_2(\widetilde{X})/\IM(\partial_{3})  \ar[d,"f \circ \text{$(\id,0)$}"] \ar[r,"\widetilde{\partial}_2"] & 
C_{1}(\widetilde{X}) \ar[d,"\text{$(\id,0)$}"] \ar[r,"\partial_{1}"] &
C_0(\widetilde{X}) \ar[d,"\id"] \\
\Z G^j \ar[r,"\text{$(\partial_{2},0)$} \circ f^{-1}"] & C_{1}(\wt X) \oplus \Z G^i \hspace{2mm} \ar[r,"\text{$(\partial_{1},0)$}"]  &
C_0(\widetilde{X})
\end{tikzcd}
\right)
\]
where $E$ is an algebraic $2$-complex over $\Z G$. Hence, by combining with \cref{lemma:chain-homotopy-equiv}, we obtain a chain homotopy equivalence $f_* \circ \varphi : C_*(X,\rho) \to E$.
\end{proof}

We now turn to the proof of \cref{thm:main0}. By \cref{prop:C_*-well-def}, $\wt C_*$ is a well-defined map and it is clear from the definition that $\wt C_*(X \vee S^n) \simeq \wt C_*(X) \oplus \Z G$ and so $\wt C_*$ gives a map of graded graphs. Note that $\Alg(G,2)$ is a tree \cite[Section 52]{Jo03a}.
In particular, $\wt C_*$ is an isomorphism of graded trees if and only if it is bijective.
There are many proofs in the literature that $\wt C_*$ is surjective (see \cite[Theorem 4]{Wa66} or \cite[Theorem 2.1]{Ma09}) and so it remains to show that $\widetilde{C}_*$ is injective.

We will now need the following two lemmas. The first is proven in the case where $X$ and $Y$ are finite 2-complexes in \cite[Proposition 2.2]{Jo03b}. The proof for finite D2 complexes is similar and so will be omitted for brevity.

\begin{lemma} \label{lemma:chain-map-1-skeleton}
Let $(X,\rho_X)$, $(Y,\rho_Y) \in \D(G)$ be such that $X^{(1)} = Y^{(1)}$. If $\nu : C_*(X,\rho_X) \to C_*(Y,\rho_Y)$ is a chain map, then $\nu$ is chain homotopy equivalent to a chain map $\varphi$ such that $\varphi \mid_{C_i(\widetilde{X})} = \id$ for $i \le 1$.
\end{lemma}

Let $\PHT(G,2) \subseteq \D(G)$ be the subgraph corresponding to the polarised homotopy types of finite $2$-complexes.

\begin{lemma}[\text{\cite[Lemma 2.3]{Jo03b}}] \label{lemma:realise-chain-map}
Let $(X,\rho_X)$, $(Y,\rho_Y) \in \PHT(G,2)$ be such that $X^{(1)} = Y^{(1)}$. If $\varphi : C_*(X,\rho_X) \to C_*(Y,\rho_Y)$ is a chain map such that $\varphi \mid_{C_i(\widetilde{X})} = \id$ for $i \le 1$, then there exists a map $f: X \to Y$ such that $f_*=\varphi_*$, $f \mid_{X^{(1)}} = \id$ and $\rho_X = \rho_Y \circ \pi_1(f)$.
\end{lemma}

\begin{proof}[Proof of Theorem \ref{thm:main0}]
Let $(X,\rho_X)$, $(Y,\rho_Y) \in \D(G)$ and note that, by the argument of \cite[Proposition 2.1]{Jo03b}, we can assume that $X^{(1)}=Y^{(1)}$ by replacing each space with a polarised homotopy equivalent space.
Suppose there is a chain homotopy 
$ \widetilde{\nu}: \widetilde{C}_*(X) \to \widetilde{C}_*(Y).$ 
By Lemma \ref{lemma:chain-homotopy-equiv}, this lifts to a chain homotopy $\nu: C_*(X,\rho_X) \to C_*(Y,\rho_Y)$ and, by Lemma \ref{lemma:chain-map-1-skeleton}, this is chain homotopy equivalent to a chain homotopy $\varphi: C_*(X,\rho_X) \to C_*(Y,\rho_Y)$ such that $\varphi \mid_{C_i(\widetilde{X})} = \id$ for $i \le 1$.

Let $i_X : X^{(2)} \hookrightarrow X$ denote the inclusion and note that this induces a $\Z G$ chain map $(i_X)_*: C_*(X^{(2)}) \to C_*(X)$ where the $2$-skeleton $X^{(2)}$ comes equipped with the polarisation $\rho_{X^{(2)}} = \rho_X \circ \pi_1(i_X)$, and similarly for $Y^{(2)}$.
Since $(\varphi \circ i_X)_3 = 0$, the composition $\varphi_* \circ (i_X)_*: C_*(X^{(2)}) \to C_*(Y)$ can be viewed as a chain map 
\[ \varphi_* \circ (i_X)_*: C_*(X^{(2)}) \to C_{* \le 2}(Y) \cong C_*(Y^{(2)}).\]
Since $(\varphi \circ i_X)_* = \id$ for $* \le 1$, Lemma \ref{lemma:realise-chain-map} implies that there exists a map $f: X^{(2)} \to Y^{(2)}$ such that $f_*=\varphi_* \circ (i_X)_*$, $f \mid_{X^{(1)}} = \id$ and $\rho_{X^{(2)}} = \rho_{Y^{(2)}} \circ \pi_1(f)$. By composing with $i_Y$, we can assume $f: X^{(2)} \to Y$ which instead has that $\rho_{X^{(2)}} = \rho_Y \circ \pi_1(f)$.

We now claim that $f$ has an extension $F: X \to Y$ such that $F_* = \varphi_*: H_2(\widetilde{X}) \to H_2(\widetilde{Y})$, which is an isomorphism since $\varphi_*$ is a homology equivalence. Since $X$ and $Y$ are finite D2 complexes, we have that $H_i(\widetilde{X}) = H_i(\widetilde{Y})=0$ for $i \ne 2$. This implies that $F$ is a homology equivalence and so is a homotopy equivalence by Whitehead's theorem. Since $F \circ i_X = f$ and $\rho_{X^{(2)}} = \rho_{Y} \circ \pi_1(f)$, this implies that $\rho_X = \rho_Y \circ \pi_1(F)$ and so $F$ is the required polarised homotopy equivalence from $(X,\rho_X)$ to $(Y,\rho_Y)$.

To find the extension $F$, first let
\[ X = X^{(2)} \cup_{\alpha_1} e_1^3 \cup_{\alpha_2} \cdots \cup_{\alpha_n} e_n^3\]
for $3$-cells $e_i^3 \cong D^3$ and attaching maps $\alpha_i \in \pi_2(X^{(2)})$, where such a decomposition exists since $X$ is assumed to be a finite $3$-complex. 

Using cellular chains, we have that $\partial_3(e_i^3) = \alpha_i$ where we are using the identification $\IM(\partial_3) \subseteq \ker(\partial_2) \cong \pi_2(X^{(2)})$, and so $\alpha_i \in \IM(\partial_3)$ for all $i = 1, \dots, n$.
Note that there is a commutative diagram
\[
\begin{tikzcd}
  \pi_2(X^{(2)}) \ar[r,"(i_X)_*"] \ar[d,"\cong"] & \pi_2(X) \ar[d,"\cong"] \\
\ker(\partial_2) \ar[r,"q"] & \ker(\partial_2)/\IM(\partial_3)
\end{tikzcd}
\]
where $q$ is the quotient map. This shows that $\IM(\partial_3) = \ker((i_X)_*)$. 
Consider the composition $f_*=\varphi_* \circ (i_X)_* : \pi_2(X^{(2)}) \to \pi_2(Y)$. Since $\varphi_*$ is a homology equivalence, this implies that $\ker((i_X)_*) = \ker(f_*)$. By combining with the above two results, we get that $\alpha_i \in \ker(f_*)$ and so the maps $f \circ \alpha_i \in \pi_2(Y)$ are nullhomotopic for all $i = 1, \dots, n$. 

By standard homotopy theory, this implies that there exists an extension $F: X \to Y$. In particular, since $f \circ \alpha_i : S^2 \to Y$ is null-homotopic, there is a map $f_i : e_i^3 \to Y$ for which $f_i \circ i = f \circ \alpha_i$ for $i : S^2 = \partial e_i^3 \hookrightarrow e_i^3$ and so we can get a well-defined map $F: X \to Y$ by defining $F \mid_{e_i^3} = f_i$ for each $i = 1, \dots, n$.
Finally note that, by the above diagram, $(i_X)_* : \pi_2(X^{(2)}) \to \pi_2(X)$ is surjective. Since $F_* \circ (i_X)_* = \varphi_* \circ (i_X)_*$ for $* \le 2$, this implies that $F_* = \varphi_*: \pi_2(X) \to \pi_2(Y)$ or, equivalently, that $F_* = \varphi_* :H_2(\widetilde{X}) \to H_2(\widetilde{Y})$.
\end{proof}

We conclude this section by noting that \cref{thm:main0} can be used to recover two well-known results on the D2 problem.
Firstly, since $\D(G)$ and $\Alg(G,2)$ are isomorphic as graphs and $\Alg(G,2)$ is a tree, this implies that $\D(G)$ is a tree. Since $\D(G)$ contains a presentation complex over $G$, this leads to the following stable solution to the D2 problem which was first proved by J. M. Cohen \cite{Co77}.

\begin{cor}
	If $X$ is a finite {\normalfont D2} complex, then there exists $n \ge 0$ for which $X \vee nS^2$ is homotopy equivalent to a finite $2$-complex.
\end{cor}

Recall that an algebraic $2$-complex over $\Z G$ is \textit{geometrically realisable} if it is chain homotopy equivalent to the cellular chain complex $C_*(X,\rho)$ of a polarised finite $2$-complex $(X,\rho)$. This begs the following question:

\begin{cproblem}[Realisation problem]
Let $G$ be a finitely presented group. Then is every algebraic $2$-complex over $\Z G$ geometrically realisable?
\end{cproblem}

Note that every algebraic $2$-complex over $\Z G$ is geometrically realisable if and only if $C_* : \PHT(G,2) \to \Alg(G,2)$ is surjective. By Theorem \ref{thm:main0}, this is the case if and only if $\PHT(G,2) = \D(G)$, i.e. if $G$ has the D2 property. This implies the following which was first proven by Johnson \cite{Jo03b} subject to a mild condition on $G$ which was subsequently shown to be unnecessary by Mannan \cite{Ma09}.

\begin{cor} \label{thm:D2=realisation}
	A finitely presented group $G$ has the {\normalfont D2} property if and only if every algebraic $2$-complex over $\Z G$ is geometrically realisable. In particular, the {\normalfont D2} problem and Realisation problem are equivalent.
\end{cor}

Finally, we remark that one can get an algebraic classification of the homotopy types of finite D2 complexes $X$ with $\pi_1(X) = G$ by quotienting $\Alg(G,2)$ by the action of $\Aut(G)$. 
Finding the number of polarised homotopy types that correspond to a given homotopy type $X$ is equivalent to finding which group automorphisms $\Aut(G)$ are induced by self-homotopy equivalences $\mathcal{E}(X)$ (see, for example, \cite{Ol65}).

\section{Projective chain complexes over integral group rings}
\label{section:preliminaries}

The aim of this section will be to recall basic preliminaries on projective chain complexes and the Swan finiteness obstruction. From now on, we will assume that $G$ is a finite group.

For $\Z G$ modules $A$ and $B$, we define $\Proj_{\Z G}^{n}(A,B)$ to be the set of chain homotopy classes of exact sequences of $\Z G$ modules
\[ E = (0 \to B \xrightarrow[]{\iota} P_{n-1} \xrightarrow[]{\partial_{n-1}} P_{n-2} \to \cdots \to P_1 \xrightarrow[]{\partial_1} P_0 \xrightarrow[]{\varepsilon} A \to 0)\]
where the $P_i$ are projective. By a chain homotopy, we formally mean a chain map which restricts to a chain homotopy equivalence on the middle terms $(P_*,\partial_*)$. For brevity, we will often omit the maps $\iota$ and $\varepsilon$ in our description of $E$.

Define the \textit{Euler class} to be
$e(E) = \sum_{i=0}^{n-1} (-1)^i [P_i] \in C(\Z G)$, which only depends on the chain homotopy type of $E$. 
For a class $\chi \in C(\Z G)$, we define $\Proj_{\Z G}^{n}(A,B;\chi)$ to be the subset of $\Proj_{\Z G}^{n}(A,B)$ consisting of those extensions with $e(E) = \chi$.
Let $\Omega_n^\chi(A)$ denote the set of $\Z G$ modules $B$ for which $\Proj_{\Z G}^{n}(A,B;\chi)$ is non-empty. 
If $n \ge 2$ then, for any $B_0 \in \Omega_n^{\chi}(A)$, we have $\Omega_n^{\chi}(A) = \{ B : B \oplus \Z G^i \cong B_0 \oplus \Z G^j, i,j \ge 0\}$, i.e. $\Omega_n^{\chi}(A)$ is a stable module \cite{Sw60b}.
We also define

\[ \Proj_{\Z G}^{n}(A,\Omega_n^\chi(A);\chi) = \bigsqcup_{B \in \Omega_n^\chi(A)} \Proj_{\Z G}^{n}(A,B;\chi)\]
which has the structure of a graded graph with grading $\chi(E) = \sum_{i=0}^{n-1} (-1)^i \rank(P_i)$ and with edges between each $E$ to the stabilised complex 
\[ E \oplus \Z G = (P_{n-1} \oplus \Z G \xrightarrow[]{(\partial_{n-1},0)} P_{n-2} \to \cdots \to P_1 \xrightarrow[]{\partial_1} P_0).\]
We can also define $\Omega_{-n}^{\chi}(B)$ to be the set of $\Z G$ modules $A$ for which $\Proj_{\Z G}^{n}(A,B;\chi)$ is non-empty and we can similarly define a graded graph $\Proj_{\Z G}^{n}(\Omega_{-n}^\chi(B),B;\chi)$.

Recall that, for a $\Z G$ module $A$, its dual is defined as $A^* = \Hom_{\Z}(A,\Z)$ which is a left $\Z G$ module under the action defined by $(g \cdot \varphi) (x) = \varphi(g^{-1} \cdot x)$ for $\varphi \in A^*$ and $g \in G$. For $\chi \in C(\Z G)$, we will write $\Omega_n^\chi(A)^* = \{ B^* : B \in \Omega_n^\chi(A)\}$.
Recall also that a $\Z G$ module $A$ is a \textit{$\Z G$ lattice} if its underlying abelian group is free, i.e. of the form $\Z^n$ for some $n \ge 0$. If $A$ is a $\Z G$ lattice, then $A^*$ is a $\Z G$ lattice and there is an isomorphism $(A^*)^* \cong A$ \cite[Section 28]{Jo03a}. If $P$ is projective, then so is $P^*$ and we also have $(P^*)^* \cong P$ since projective $\Z G$ modules are $\Z G$ lattices. 

If $E = (P_*,\partial_*) \in \Proj^n_{\Z G}(A,B)$, then define
\[ E^* = (P_{0}^* \xrightarrow[]{\partial_{1}^*} P_{1}^* \to \cdots \to P_{n-2}^* \xrightarrow[]{\partial_{n-1}^*} P_{n-1}^*).\]

\begin{lemma}[\text{\cite[Proposition 28.4]{Jo03a}}] \label{lemma:dual-of-proj}
Let $A$ and $B$ be $\Z G$ lattices. Then we have $E^* \in \Proj^n_{\Z G}(B^*,A^*)$ and $(E^*)^* \simeq  E$ are chain homotopy equivalent.
\end{lemma}

By addition of elementary complexes, we can show that every projective extension $E$ with $e(E) = 0$ is chain homotopy equivalent to an extension with the $P_i$ free provided $n \ge 2$. In particular, $\Proj_{\Z G}^{n}(A,B;0)$ can be taken to be the set of chain homotopy types of exact sequences $E$ with the $P_i$ free. We also let $\Omega_n(A) = \Omega_n^0(A)$.

It was shown by Swan \cite[Theorem 4.1]{Sw60b} that a group $G$ has $n$-periodic cohomology if and only if there exists an exact sequence of the form
\[ 0 \to \Z \to P_{n-1} \to P_{n-2} \to \cdots \to P_1 \to P_0 \to \Z \to 0\]
where the $P_i$ are projective, i.e. if $\Proj^n_{\Z G}(\Z,\Z)$ is non-empty. By taking the map $P_0 \to P_{n-1}$ which factors through $\Z$, this can be turned into in $n$-periodic projective resolution. We say that $G$ has \textit{free period $n$} if such a resolution exists with the $P_i$ free or, equivalently, if $\Proj^n_{\Z G}(\Z,\Z;0)$ is non-empty. 

Recall that the \textit{Swan map} $S: (\Z / |G|)^\times \to C(\Z G)$
sends $r \mapsto [(r,\Sigma)]$, where $\Sigma = \sum_{g \in G} g$ is the group norm and $(r,\Sigma) \subseteq \Z G$ has finite index coprime to $|G|$ and so is projective by \cite{Sw60a}. We define the \textit{Swan subgroup} to be $T_G = \IM(S)$. Finally, if $G$ has $n$-periodic cohomology and $E \in \Proj^n_{\Z G}(\Z,\Z)$, then we define the \textit{Swan finiteness obstruction} to be the class
\[ \sigma_n(G) = [e(E)] \in C(\Z G)/T_G.\]

It was shown by Swan that $\sigma_n(G)$ is independent of the choice of $E$ \cite[Lemma 7.3]{Sw60b}. Furthermore, we have:

\begin{thm}[\text{\cite{Sw60b}}] \label{thm:swan-finiteness-obstruction}
Let $G$ have $n$-periodic cohomology. Then the following are equivalent:
\begin{enumerate}[\normalfont(i)]
\item $G$ has free period $n$.
\item $\sigma_n(G) = 0 \in C(\Z G)/T_G$.
\item There is a finite 	CW-complex $X$ such that $X \simeq S^{n-1}$ and $G$ acts freely on $X$.
\end{enumerate}
\end{thm}

It took over 20 years until the first example of a group with $\sigma_n(G) \ne 0$ was found by R. J. Milgram \cite{Mi85}. It was later shown by J. F. Davis \cite{Da82} that the group $Q(16;3,1)$ with $4$-periodic cohomology has free period $8$, which is the example of minimal order. Conversely, we also have:

\begin{prop}[\text{\cite[Lemma 7.4]{Sw83}}] \label{prop:P_G-arises}
Let $G$ have $n$-periodic cohomology and let $P_G$ be a projective $\Z G$ module for which $\sigma_n(G) = [P_G] \in C(\Z G)/T_G$. Then there exists $E \in \Proj^n_{\Z G}(\Z,\Z)$ for which $e(E) = [P_G] \in C(\Z G)$.	
\end{prop}

The formulation (iii) has the following consequence for finite Poincar\'{e} $3$-complexes $X$ since, if $\pi_1(X)$ is finite, then $\wt X \simeq S^3$.

\begin{cor} \label{cor:pi_1-of-poinc}
A finite group $G$ is the fundamental group of a finite Poincar\'{e} $3$-complex if and only if $G$ has free period $4$.
\end{cor}

\section{Classification of algebraic 2-complexes} 
\label{section:D2-overview}

This section will largely be dedicated to the proof of the following theorem from the introduction.

\begingroup
\begin{thm} \label{thm:main1}
Let $G$ have $4$-periodic cohomology. Then there is an isomorphism of graded trees
\[ \Psi: \Alg(G,2) \to [P_G]\]
for any projective $\Z G$ module $P_G$ for which $\sigma_4(G) = [P_G] \in C(\Z G)/T_G$.
\end{thm}
\endgroup

A similar statement appears in \cite[Theorem 57.4]{Jo03a} though, due to a gap in the proof of Theorem 56.9, the argument given only applies in the case of minimal algebraic 2-complexes. Furthermore, the argument that $\{E \in \Alg(G,2) : \chi(E) = r\}$ and $\{P \in [P_G] : \rank(P)=r\}$ are in one-to-one correspondence for all $r \ge 1$ assumes that the former set is non-empty for each $r$, and this was only subsequently shown when $G$ has free period 4 \cite[p234]{Jo03a}.
 
Let $\mu_2(G)$ be the minimum value of $\chi(E)$ over all $E \in \Alg(G,2)$. Since $G$ is finite, we have $\mu_2(G) \ge 1$ \cite[Corollary 1.3]{Sw65}. We begin by showing the following:

\begin{prop}\label{thm:mu2=1}
If $G$ has $4$-periodic cohomology, then $\mu_2(G) = 1$.
\end{prop}

Recall that the \textit{augmentation ideal} is the module $I = \ker(\varepsilon: \Z G \to \Z)$ where $\varepsilon: \Z G \to \Z$ sends $g \mapsto 1$ for all $g \in G$. We will need the following three lemmas.

\begin{lemma} 
\label{lemma:1st-syzygy}
Let $G$ be a finite group and let $\alpha: \Z \to \Z G^n$ be injective. If $\coker(\alpha)$ is a $\Z G$ lattice, then $\coker(\alpha) \cong I^* \oplus \Z G^{n-1}$ where $I$ is the augmentation ideal.
\end{lemma}

\begin{proof}
The case $n=1$ follows from the fact that $\alpha = r\Sigma$ for $r \ne 0$ since $\Z G/\Sigma \cong I^*$, and the case $n \ge 2$ follows from \cite[Proposition 29.2]{Jo03a}.
\end{proof}

\begin{lemma} \label{lemma:replacement}
Let $P$ be a projective $\Z G$-module, let $r = \rank(P)$ and let $I$ be the augmentation ideal. Then there exists a $\Z G$ lattice $J$ for which 
\[ I \oplus P = J \oplus \Z G^r.\]
\end{lemma}

\begin{proof}
Since $G$ is a finite, \cite[Theorem A]{Sw60a} implies that $P$ is of the form $P = P_0 \oplus \Z G^{r-1}$ for some rank one projective $\Z G$ module $P_0$ and so it suffices to prove the case $r=1$. 
Let $\varphi : P \to \Z$ be the surjection obtained by taking the composition 
\[ 
P \to P \otimes \Q \cong \Q G \to \Q
\]
whose image is a non-trivial finitely-generated subgroup of $\Q$, and so isomorphic to $\Z$. If $J = \ker(\varphi)$ then, by applying Schanuel's lemma to the exact sequences
\[ 0 \to I \to \Z G \to \Z \to 0, \qquad 0 \to J \to P \to \Z \to 0,\]
we get that $I \oplus P \cong J \oplus \Z G$.
\end{proof}

\begin{lemma}[\text{\cite[p514]{Wa79b}}] \label{lemma:ext}
Let $J$ be a $\Z G$ lattice. Then $\Ext^1_{\Z G}(J, \Z G)=0$.
\end{lemma}

\begin{proof}[Proof of \cref{thm:mu2=1}]
Since $G$ has $4$-periodic cohomology, the discussion in Section \ref{section:preliminaries} implies that there exists an exact sequence of $\Z G$-modules
\[ 0 \to \Z \xrightarrow{\alpha} F_3 \to P_2 \to F_1 \to F_0 \to \Z \to 0 \]
where $P_2$ is projective and, by addition of elementary complexes, we can assume the $F_i$ are free.
By Lemma \ref{lemma:1st-syzygy}, $\coker(\alpha) \cong I^* \oplus \Z G^r$ where $r=\rank(F_3)-1$. This gives an exact sequence:
\[ 0 \to I^* \oplus \Z G^r \to P_2 \to F_1 \to F_0 \to \Z \to 0.\]
Now let $\bar{P}_2$ be a projective for which $F_2=P_2 \oplus \bar{P}_2$ is free. By forming the direct sum with length two exact sequence, we get
\[ 0 \to I^* \oplus P \to F_2 \to F_1 \to F_0 \to \Z \to 0\]
for $P=\Z G^r \oplus \bar{P}_2$ projective. By dualising the result in Lemma \ref{lemma:replacement}, we can write $I^* \oplus P = J \oplus \Z G^s$ where $s = \rank(P)$.

Let $i$ denote the injection $i: J \oplus \Z G^s \cong I^* \oplus P \to F_2$, and consider the exact sequences:
\[ 0 \to J \to F_2/i(\Z G^s) \to F_1 \to F_0 \to \Z \to 0, \quad 0 \to \Z G^s \to F_2 \to F_2/i(\Z G^s) \to 0.\]
The first exact sequence implies that $F_2/i(\Z G^s)$ is a $\Z G$ lattice. By Lemma \ref{lemma:ext}, this implies that 
$\Ext^1_{\Z G}(F_2/i(\Z G^s),\Z G^s) = 0$ and so $F_2 \cong F_2/i(\Z G^s) \oplus \Z G^s$ by the second exact sequence. Hence we get an exact sequence 
\[ 0 \to J \to F_2 \to F_1 \oplus \Z G^s \to F_0 \to \Z \to 0\]
which defines an algebraic $2$-complex $E$. Since $J \cong I^* \cong \Z G^{|G|-1}$ as abelian groups, we have that $\chi(E)=1$ which completes the proof.
\end{proof}

Recall that a graded tree is a \textit{fork} if it has a single vertex at each non-minimal level (i.e. grade). The following was shown by W. H. Browning. These results were never published, though an alternate proof can be found in \cite[Corollary 2.6]{Ha19}.

\begin{thm}[\text{\cite[Theorem 5.4]{Br78}}]
Let $G$ be a finite group. Then $\Alg(G,2)$ is a fork.
\end{thm}

On the other hand, if $G$ is a finite group, then \cite[Theorem A]{Sw60a} implies that every projective $\Z G$ module $P$ is of the form $P = P_0 \oplus \Z G^r$ for some rank one projective $\Z G$ module $P_0$. This implies that $[P]$ is also a fork. 

Hence, in order to prove \cref{thm:main1}, it suffices to prove that there is a bijection $\Psi$ between $\Alg(G,2)$ and $[P_G]$ at the minimal levels, i.e. that there is a one-to-one correspondence between $\{E \in \Alg(G,2) : \chi(E)=1\}$ and $\{ P \in [P_G] : \rank(P)=1\}$ (see \cref{figure:fork}). We now need the following two results.

\begin{figure}[h] \vspace{-4mm} 
\begin{center}
\begin{tabular}{c c c}
\begin{tabular}{l}
\begin{tikzpicture}
\draw[fill=black] (0,0) circle (2pt);
\draw[fill=black] (1,0) circle (2pt);
\draw[fill=black] (2,0) circle (2pt);
\draw[fill=black] (3,0) circle (2pt);
\draw[fill=black] (4,0) circle (2pt);
\draw[fill=black] (2,1) circle (2pt);
\draw[fill=black] (2,2) circle (2pt);
\draw[fill=black] (2,3) circle (2pt);
\node at (2,3.6) {$\vdots$};
\draw[thick] (0,0) -- (2,1) (1,0) -- (2,1) (2,0) -- (2,1) (4,0) -- (2,1) -- (2,2) -- (2,3);
\draw[dashed] (3,0) -- (2,1);
\end{tikzpicture} 
\end{tabular}
&
\begin{tabular}{l}
\hspace{-3mm}$\xrightarrow[]{\quad \Psi \quad}$\hspace{-3mm}
\end{tabular}
&
\begin{tabular}{l}
\begin{tikzpicture}
\draw[fill=black] (0,0) circle (2pt);
\draw[fill=black] (1,0) circle (2pt);
\draw[fill=black] (2,0) circle (2pt);
\draw[fill=black] (3,0) circle (2pt);
\draw[fill=black] (4,0) circle (2pt);
\draw[fill=black] (2,1) circle (2pt);
\draw[fill=black] (2,2) circle (2pt);
\draw[fill=black] (2,3) circle (2pt);
\node at (2,3.6) {$\vdots$};
\draw[thick] (0,0) -- (2,1) (1,0) -- (2,1) (2,0) -- (2,1) (3,0) -- (2,1) (2,1) -- (2,2) -- (2,3);
\draw[dashed] (4,0) -- (2,1);
\end{tikzpicture}
\end{tabular}
\end{tabular}
\end{center}
\caption{Tree structures for $\Alg(G,2)$ and $[P_G]$} \label{figure:fork}
\vspace{-2mm}
\end{figure}

\begin{prop} \label{prop:alg-to-stab}
There is an isomorphism of graded trees
\[\Alg(G,2) \cong \Proj_{\Z G}^{3}(\Z, \Omega_3(\Z);0). \]
\end{prop}

\begin{prop} \label{prop:stab1-to-modules}
Let $\chi = [P] \in C(\Z G)$. Then there is an map of graded trees
	\[ \Phi: \Proj_{\Z G}^{1}(\Omega_3(\Z), \Z;\chi) \to [P] \]
	given by $(0 \to \Z \to P_0	\to J \to 0) \mapsto P_0$, which is a bijection at the minimal level.
\end{prop}

The first is immediate from the discussion in \cref{section:preliminaries}, and the second is a consequence of the following which is a slight extension of \cite[Corollary 56.5]{Jo03a}.

\begin{lemma} \label{lemma:aug-seq}
For $i=1,2$, let $P_i$ be projective $\Z G$ modules of rank one and let $\mathcal{E}_i = (0 \to J \to P_i	\to \Z \to 0)$ be exact sequences of $\Z G$-modules. Then there is a chain homotopy equivalence $\mathcal{E}_1 \simeq \mathcal{E}_2$ if and only if $P_1 \cong P_2$.
\end{lemma}

The following will be useful in comparing Propositions \ref{prop:alg-to-stab} and \ref{prop:stab1-to-modules}.

\begin{lemma} \label{lemma:duality-chain}
Let $G$ have $4$-periodic cohomology and let $P_G$ be a projective $\Z G$ module for which $\sigma_4(G) = [P_G] \in C(\Z G)/T_G$. If $\chi = [P_G^*] \in C(\Z G)$, then 
\[\Omega_3(\Z) = \Omega_1^\chi(\Z)^*.\]
\end{lemma}

\begin{proof}
By \cref{prop:P_G-arises}, there exists $E \in \Proj^4_{\Z G}(\Z,\Z)$ with $e(E) = [P_G]$. By addition of elementary complexes, we can assume that
\[ E = (P \xrightarrow{\partial_3} F_2 \xrightarrow{\partial_2} F_1 \xrightarrow{\partial_1} F_0)\]
where the $F_i$ are free, so that $P \in [P_G]$.  

Let $J = \ker(\partial_2) = \IM(\partial_3)$. Then $J \in \Omega_3(\Z)$ and there are exact sequences 
\[ \mathcal{E} =( 0 \to \Z \xrightarrow[]{\alpha} P \xrightarrow[]{\beta} J \to 0), \quad \mathcal{E}^* = (0 \to J^* \xrightarrow[]{\beta^*} P^* \xrightarrow[]{\alpha^*} \Z \to 0), \]
where $\mathcal{E}^*$ is exact by \cref{lemma:dual-of-proj} since $\Z$ and $J$ are $\Z G$ lattices.  Hence $J^* \in \Omega_1^\chi(\Z)$. Since $(J^*)^* \cong J$, this implies that $J \in \Omega_1^\chi(\Z)^*$. Hence $\Omega_3(\Z)=\Omega_1^\chi(\Z)^*$ since two stable modules are equal if they intersect non-trivially.
\end{proof}

Recall that, if $J$ is a $\Z G$ module, then an automorphism $\varphi: J \to J$ induces a map $\varphi_*: H^{n}(G;J) \to H^{n}(G;J)$. If $J \in \Omega_n(\Z)$, then $H^{n}(G;J) \cong \Z / |G|$ \cite[p132]{Jo03a}. By fixing this identification, the construction $\varphi \mapsto \varphi_*$ induces a map
\[ \nu^J : \Aut_{\Z G}(J) \to (\Z /|G|)^\times.\]
Let $S: (\Z /|G|)^\times \to C(\Z G)$ denote the Swan map, as defined in Section \ref{section:preliminaries}. Then:

\begin{lemma}[\text{\cite[Theorems 54.6, 56.10]{Jo03a}}] \label{lemma:stab-classify}
Let $\chi \in C(\Z G)$ and $J \in \Omega_n^\chi(\Z)$. Then  $\IM(\nu^J) \subseteq \ker(S)$ and there is a bijection
	\[ \Proj_{\Z G}^{n}(\Z, J;\chi) \cong \ker(S)/\IM(\nu^J).\]
\end{lemma}

In particular, $\Proj_{\Z G}^{n}(\Z, J;\chi)$ only depends on $J$ and not on $n$ or $\chi$. 

\begin{proof}[Proof of Theorem \ref{thm:main1}]
First note, since the map $P \mapsto P^*$ is an involution on the class of projective $\Z G$ modules, it must induce an isomorphism of graded trees $[P_G] \cong [P_G^*]$. Hence, by Propositions \ref{prop:alg-to-stab} and \ref{prop:stab1-to-modules}, it suffices to prove that the graded trees
\[ \Proj_{\Z G}^{3}(\Z, \Omega_3(\Z);0), \quad \Proj_{\Z G}^{1}(\Z, \Omega_1^{\chi}(\Z);\chi) \]
contain the same number of extensions at the minimal level, where $\chi = [P_G^*]$.

To see this, let $J \in \Omega_3(\Z)$ be minimal and note that $\Aut_{\Z G}(J) \cong \Aut_{\Z G}(J^*)$ and so there is a bijection $\IM(\nu^J) \cong \IM(\nu^{J^*})$. In particular, we have bijections
	\[ \Proj_{\Z G}^{3}(\Z, J;0) \cong \ker(S)/\IM(\nu^J) \simeq \ker(S)/\IM(\nu^{J^*}) \cong \Proj_{\Z G}^{1}(\Z, J^*;\chi).\]
By \cref{lemma:duality-chain}, the map $J \mapsto J^*$ induces a bijection $\Omega_3(\Z) \cong \Omega_1^{\chi}(\Z)$. 
We can now extend the bijection $\Proj_{\Z G}^{3}(\Z, J;0) \cong \Proj_{\Z G}^{1}(\Z, J^*;\chi)$ over all $J \in \Omega_3(\Z)$ at the minimal level, and this completes the proof.
\end{proof}

We conclude this section by remarking that, whilst it will not be needed in the proof of \cref{thm:D2}, we can obtain the following explicit form for the map $\Psi$.

\begin{prop} \label{prop:Psi-explicit}
Let $G$ have $4$-periodic cohomology, let $\bar{E} \in \Proj^4_{\Z G}(\Z,\Z)$ and let $P_G$ be a projective $\Z G$ module for which $e(\bar{E}) = [P_G]$.
Then there is an isomorphism of graded trees
\[ \Psi: \Alg(G,2) \to [P_G]\]
sending $E = (F_*, \partial_*) \mapsto P$, where $P$ is the unique projective $\Z G$ module for which
\[ (0 \to \Z \xrightarrow[]{\alpha} P \xrightarrow[]{\beta} F_2 \xrightarrow[]{\partial_2} F_1 \xrightarrow[]{\partial_1} F_0 \xrightarrow[]{\partial_0} \Z \to 0) \simeq \bar{E} \]
is a chain homotopy equivalence for some $\alpha$ and $\beta$.
\end{prop}

\begin{proof}
It follows from \cite[Lemma 1.1]{Wa79b} that, for each $E =(F_*,\partial_*)$ and $J = \ker(\partial_2) \subseteq F_2$, there is a bijection
\[ \varphi_{E} : \Proj^1_{\Z G}(J,\Z) \to \Proj^4_{\Z G}(\Z,\Z)\]
\vspace{-4mm}
\[ (0 \to \Z \xrightarrow[]{\alpha} P \xrightarrow[]{\beta} J \to 0) \mapsto (0 \to \Z \xrightarrow[]{\alpha} P \xrightarrow[]{\beta} F_2 \xrightarrow[]{\partial_2} F_1 \xrightarrow[]{\partial_1} F_0 \xrightarrow[]{\partial_0} \Z \to 0). \]

Hence, for $\chi = [P_G]$, we have well-defined maps of graded trees
\[ \Psi : \Alg(G,2) \to \Proj^1_{\Z G}(\Omega_3(\Z),\Z;\chi) \to [P_G] \]
where the first map is given by sending $E \mapsto \varphi_{E}^{-1}(\bar{E})$, and the second map sends $(0 \to \Z \xrightarrow[]{\alpha} P \xrightarrow[]{\beta} J \to 0) \mapsto P$ which is contained in $[P_G]$ since $e(\bar{E}) = [P_G]$ only depends on the chain homotopy type of $\bar{E}$.
It follows by tracing through the proof of \cref{thm:main1} that this is a bijection and so is an isomorphism of graded trees.
\end{proof}

In a subsequent article \cite{Ni20}, we will explore the correspondence given in \cref{prop:Psi-explicit} in more detail and in a more general form.

\section{Cancellation for projective modules over integral group rings} \label{section:projective-modules}

Recall from algebraic K-theory that $K_0(\Z G)$ is defined as the Grothendieck group of the monoid of isomorphism classes of projective $\Z G$-modules $P(\Z G)$, i.e. the abelian group generated by $[P]$ for $P \in P(\Z G)$ with relations $[P_1 \oplus P_2] = [P_1] \oplus [P_2]$. Recall also that $K_1(\Z G) = \GL(\Z G)^{\text{ab}}$ where $\GL(\Z G) = \bigcup_n \GL_n(\Z G)$. 

For a projective $\Z G$ module $P$, let $\Aut(P)$ denote the group of $\Z G$ module automorphisms of $P$ and define $\Aut(P) \to K_1(\Z G)$ by choosing a projective $Q$ such that $P \oplus Q \cong \Z G^r$ is free and then forming the composition
\[ \Aut(P) \subseteq \Aut(P \oplus Q) \cong \GL_r(\Z H) \subseteq \GL(\Z H) \to K_1(\Z H).\]
This is well-defined by \cite[Lemma 3.2]{Mi71}.

The aim of this section will be to prove the following. 

\begingroup
\begin{thm} \label{thm:main2}
Let $G$ be a finite group with quotient $H=G/N$, let $\bar{P} \in \LF_1(\Z G)$ and let $P = \bar{P} \otimes_{\Z N} \Z \in \LF_1(\Z H)$. Suppose $m_{\H}(G)=m_{\H}(H)$ and that the map 
\[ \Aut (P) \to K_1(\Z H)\]
is surjective. Then $[\bar{P}]$ has cancellation if and only if $[P]$ has cancellation.
\end{thm}
\endgroup

Our proof will follow a similar outline to \cite[Theorem A]{Ni21}, which established this result in the special case where $\bar{P}=\Z G$ and $P=\Z H$.

\subsection{Preliminaries on locally free modules}

From now on, we will let $A$ be a finite-dimensional semisimple $\Q$-algebra and we will let $\Lambda$ be a $\Z$-order in $A$, i.e. a finitely-generated subring of $A$ such that $\Q \cdot \Lambda = A$. For example, we can take $\Lambda = \Z G$ and $A = \Q G$.
For a prime $p$, let $\Lambda_p = \Lambda \otimes \Z_p$ where $\Z_p$ is the $p$-adic integers. We say that a $\Lambda$ module $M$ is \textit{locally free of rank $n$} if $M \otimes \Z_p$ is a free $\Lambda_p$ module of rank $n$ for all primes $p$. 

Let $C(\Lambda)$ denote the locally free class group, i.e. the equivalence classes of locally free $\Lambda$ modules where $P_1 \sim P_2$ if $P_1 \oplus \Lambda^i \cong P_2 \oplus \Lambda^j$ for some $i,j \ge 0$, and we let $\LF_1(\Lambda)$ denote the set of isomorphism classes of locally free $\Lambda$ modules of rank one. We say that $\l$ has \textit{locally free cancellation} if $[P]$ has cancellation for all locally free $\l$ modules $P$.
By \cite[Lemma 2.1]{Sw80}, every locally free $\Lambda$ module is projective. The converse holds in the case where $\Lambda = \Z G$ by \cite[Theorems 2.21, 4.2]{Sw70} and so $C(\Z G)$ and $\LF_1(\Z G)$ coincide with our previous definitions. 

Let $m_{\H}(\l)$ denote the number of copies of $\H$ in the Wedderburn decomposition of $\l_{\R} = \l \otimes \R$, and we say that $\l$ satisfies the \textit{Eichler condition} if $m_{\H}(\l)=0$. The following two results will be essential to the proof of \cref{thm:main2}.

\begin{thm}[\text{\cite[Theorem 4.1]{Ja68}}] \label{thm:jacobinski}
Suppose $\Lambda$ satisfies the Eichler condition. Then $\Lambda$ has locally free cancellation.
\end{thm}

\begin{thm}[\text{\cite[Corollary 10.5]{Sw80}}] \label{thm:swan-K_1}
Suppose $\Lambda$ satisfies the Eichler condition, $I \subseteq \l$ is a two-sided ideal of finite index and $f:\Lambda \to \Lambda/I$ is the induced map. Then $f(\Lambda^\times) \unlhd (\Lambda/I)^\times$ and $(\Lambda/I)^\times \to K_1(\Lambda/I)$ induces an isomorphism
	\[ \frac{(\Lambda/I)^\times}{\Lambda^\times} \cong \frac{K_1(\Lambda/I)}{K_1(\Lambda)}. \]
\end{thm}

Recall the following which is a refinement of a result of A. Fr\"{o}hlich \cite{Fr75}. This gives the forward direction of \cref{thm:main2}.

\begin{thm}[\text{\cite[Theorem A10]{Sw83}}] \label{thm:cancellation-closed-under-quotients}
Let $H=G/N$, let $\bar{P} \in \LF_1(\Z G)$ and let $P = \bar{P} \otimes_{\Z N} \Z \in \LF_1(\Z H)$. If $[\bar{P}]$ has cancellation, then $[P]$ has cancellation.
\end{thm}

Finally, recall that every projective $\Z G$ module $P$ is of the form $P = P_0 \oplus \Z G^r$ where $P_0$ has rank one \cite[Theorem A]{Sw60a}. In particular, the stable class map $[\, \cdot \,] : \LF_1(\Z G) \to C(\Z G)$ is surjective and is bijective precisely when $\Z G$ has locally free cancellation. Furthermore, $\Z G$ has cancellation in the class of $[P]$ precisely when the fibre over $[P] \in C(\Z G)$ is trivial.

\subsection{Proof of \cref{thm:main2}}

Let $H=G/N$ where $N \le G$ is a normal subgroup, let $\bar{P} \in \LF_1(\Z G)$ and let $P = \bar{P} \otimes_{\Z N} \Z \in \LF_1(\Z H)$.
Since the converse direction was proven in Theorem \ref{thm:cancellation-closed-under-quotients}, it will suffice to prove that $[P]$ has cancellation subject to the following three conditions:
\begin{enumerate}[\normalfont (i)]
\item $[\bar{P}]$ has cancellation.
\item $m_{\H}(G) = m_{\H}(H)$.
\item The map $\Z H^\times \to K_1(\Z H)$ is surjective.
\end{enumerate}

We begin by considering the following pullback diagram for $\Z G$ induced by $N$:
\[
\begin{tikzcd}
  \Z G \arrow[r,"i_2"] \arrow[d,"i_1"] & \Lambda \arrow[d,"j_2"] \\
  \Z H \arrow[r,"j_1"] & (\Z/n \Z)[H]
\end{tikzcd}
\]
where $n = |N|$, $\wh N = \sum_{g \in N} g$ and $\Lambda = \Z G / \wh N$. This is the standard pullback diagram for the ring $\Z G$ and trivially intersecting ideals 
\[I= \ker( \Z G \to \Z H) = I(N) \cdot G \quad \text{and} \quad J= \widehat{N} \cdot \Z G,\] 
where $I(N)= \ker(\Z N \to \Z)$ is the augmentation ideal \cite[Example 42.3]{CR90}. 

Note that $\Lambda$ is a $\Z$-order in $\Lambda_\Q$ which is a finite-dimensional semisimple $\Q$-algebra since, by tensoring the diagram with $\Q$, we get that $\Q G \cong \Q H \times \Lambda_\Q$. By tensoring further with $\R$, we get that $\R G \cong \R H \times \Lambda_{\R}$ and so $m_{\H}(G) = m_{\H}(H) + m_{\H}(\l)$. 

By condition (ii), $\Lambda$ satisfies the Eichler condition and so $\Lambda$ has locally free cancellation by Theorem \ref{thm:jacobinski}.
 In particular, $\LF_1(\Lambda) = C(\Lambda)$ and, by condition (i), we also have that $\LF_1(\Z H) \to C(\Z H)$ has trivial fibre over $[P]$. 

Consider the following diagram induced by the maps on projective modules.
\begin{equation*}
\begin{tikzcd}
\LF_1(\Z G) \arrow[r,"\varphi_1"] \arrow[d, twoheadrightarrow] & \LF_1(\Z H) \times \LF_1(\Lambda) \arrow[d,twoheadrightarrow] \\
C(\Z G) \arrow[r,"\varphi_2"] & C(\Z H) \times C(\Lambda)
\end{tikzcd}
\end{equation*}
Let $Q$ be the image of $\bar{P}$ in $\LF_1(\Lambda)$. Then proving $[\bar{P}]$ has cancellation amounts to proving that the fibres of $\varphi_1, \varphi_2$ over $(P,Q)$ are in bijection. 
We will now compute each of these fibres in turn.

Firstly, by Theorem 8.1 of \cite{Sw60a}, we know that 
\[P \otimes (\Z/n\Z)[H] \cong (\Z/n\Z)[H] \cong Q \otimes (\Z/n\Z)[H].\] 
The pullback diagram for $\Z G$ above is a Milnor square and so, by the general construction of projectives modules using a Milnor square (see, for example, \cite[Proposition 4.1]{Sw80}), the fibre $\varphi_1^{-1}(P,Q)$ is in correspondence with the double coset 
\[ \Aut (P) \backslash (\Z/n\Z)[H]^\times / \Aut (Q)\] 
which we take to mean $\bar{j}_1(\Aut(P)) \backslash (\Z/n\Z)[H]^\times / \bar{j}_2(\Aut (Q))$ where $\bar{j}_1$ is the induced map 
$\Aut(P) \to \Aut(P \otimes (\Z/n\Z)[H]) \cong (\Z/n\Z)[H]^\times$ and similarly for $\bar{j}_2$.

Secondly, by \cite[Theorem 3.3]{Mi71}, there is an exact sequence of the form
\[
\begin{tikzcd}
  K_1(\Z H) \times K_1(\Lambda) \rar 
    & K_1 ((\Z/n\Z)[H]) \arrow[r,"\partial"]  &   
  K_0(\Z G) \arrow[r] & K_0(\Z H) \times K_0(\Lambda)
\end{tikzcd}
\]
where all maps other than $\partial$ are functorial, which is a part of the Mayer-Vietoris sequence for the Milnor square above.

Since projective $\Z G$ modules are locally free, the locally free rank induces a surjection $\rk : K_0(\Z G) \to \Z$ and, by \cite[p157]{Sw80}, we have that $C(\Z G) \cong \ker(\rk)$. It is now straightforward to check that $\ker(\varphi_2) \cong \ker(K_0(\Z G) \to K_0(\Z H) \times K_0(\Lambda))$. Hence, by exactness, the fibres of $\varphi_2$ are in one-to-one correspondence with
\[ \text{Im}(K_1 ((\Z/n\Z)[H]) \to C(\Z G)) \cong \frac{K_1((\Z/n\Z)[H])}{K_1(\Z H) \times K_1(\Lambda)} \]
where we take the denominator to mean $\bar{j_1}(K_1(\Z H)) + \bar{j_2}(K_1(\l))$ where $\bar{j}_i = K_1(j_i)$.

Now note that $\End(Q)$ satisfies the Eichler condition since 
\[\End(Q) \otimes \R \cong \End(Q \otimes \R) \cong \End(\Lambda_\R) \cong \Lambda_\R. \]
Let $J = \ker(\Lambda \twoheadrightarrow (\Z/n\Z)[H])$ and note that there is a map $\End (Q) \twoheadrightarrow \Lambda/J$ induced by localisation \cite[p146]{Sw83}. 
This implies that $I=\ker(\End (Q) \twoheadrightarrow \Lambda/J)$ has finite index in $\End (Q)$ and is such that 
$ \End(Q)/I \cong \Lambda/J \cong (\Z/n\Z)[H].$ 

Hence we can apply \cref{thm:swan-K_1} to get that
\[ \frac{(\Z/n\Z)[H]^\times}{\Aut(Q)} \cong \frac{K_1((\Z/n\Z)[H])}{K_1(\End(Q))}. \]
by using that $\Aut(Q) = \End(Q)^\times$. Since there is a commutative diagram
\begin{equation*}
\begin{tikzcd}
 K_1(\Lambda) \ar[d,"\cong"] \arrow[r] & K_1(\Lambda/J) \\
K_1(\End (Q)) \arrow[ur]
\end{tikzcd}
\end{equation*}
by \cite[Corollary A17]{Sw83}, we get that 
\[ A=\frac{(\Z/n\Z)[H]^\times}{\Aut (Q)} \cong \frac{K_1((\Z/n\Z)[H])}{K_1(\Lambda)} \]
and so $\Aut (P) \backslash (\Z/n\Z)[H]^\times / \Aut (Q)$ and $\frac{K_1((\Z/n\Z)[H])}{K_1(\Z H) \times K_1(\Lambda)}$ are in correspondence if and only if the maps $\Aut (P) \to A$ and $K_1(\Z H) \to A$ have the same images. 

Since we have a map $\Aut (P) \to \Aut_{(\Z/n\Z)[H]}(P \otimes (\Z/n\Z)[H]) \cong (\Z/n\Z)[H]^\times$, we can obtain the following commutative diagram.

\begin{equation*}
\begin{tikzcd}
\Aut (P) \arrow[r] \arrow[d,"\varphi"] & (\Z/n\Z)[H]^\times \arrow[r,twoheadrightarrow] \arrow[d] & \frac{(\Z/n\Z)[H]^\times}{\Aut (Q)} \cong A \arrow[d,equals] \\
K_1(\Z H) \arrow[r] & K_1((\Z/n\Z)[H]) \arrow[r,twoheadrightarrow] & \frac{K_1((\Z/n\Z)[H])}{K_1(\Lambda)} \cong A
\end{tikzcd}
\end{equation*}

To see that the left hand square commutes, suppose $\varphi: \Aut(P) \to K_1(\Z H)$ is defined via $P \oplus P' \cong_f \Z G^r$ for some $r \ge 1$. Since $(P \oplus P') \otimes (\Z/n\Z)[H] \cong (\Z/n\Z)[H]^r$ and $\Z G^r \otimes (\Z/n\Z)[H] \cong (\Z/n\Z)[H]^r$, this induces an automorphism
\[ F=f \otimes (\Z/n\Z)[H] : (\Z/n\Z)[H]^r \to (\Z/n\Z)[H]^r\]
from which we can define a map
\[ (\Z/n\Z)[H] \hookrightarrow (\Z/n\Z)[H] \oplus (P' \otimes (\Z/n\Z)[H]) \cong (\Z/n\Z)[H]^r \xrightarrow{F} (\Z/n\Z)[H]^r.\]
Since this induces a map of units $(\Z/n\Z)[H]^\times \to \GL_r((\Z/n\Z)[H])$, we can use this to define a map $(\Z/n\Z)[H]^\times \to K_1((\Z/n\Z)[H])$. It follows from \cite[Lemma 3.2]{Mi71} that this is the same as the map defined using the inclusion
\[(\Z/n\Z)[H]^\times \hookrightarrow \GL_r((\Z/n\Z)[H])\]
and so is the same as the middle vertical map in the diagram above. The left hand square then commutes by construction of the map $F$.

If $\psi_1: \Aut(P) \to A$ denotes the map along the top row and $\psi_2: K_1(\Z H) \to A$ denotes the map along the bottom row, then commutativity shows that $\psi_1=\psi_2 \circ \varphi$. 
By condition (iii), $\varphi$ is surjective. Hence $\IM \psi_1 = \IM \psi_2$ and so $[P]$ has cancellation. This completes the proof of Theorem \ref{thm:main2}.

\section{Groups with periodic cohomology} \label{section:groups-with-periodic-cohomology}

The aim of this section will be to find restrictions on the quotients of groups with periodic cohomology which will allow us to apply Theorem \ref{thm:main2} in Section \ref{section:cancellation-periodic-groups}. We will also classify the groups $G$ with $4$-periodic cohomology for which $m_{\H}(G) \le 2$, from which we will obtain the list (i)-(vii) stated in the introduction.

Recall that a group $G$ has \textit{$n$-periodic cohomology} if the Tate cohomology groups $\wh H^{*}(G;\Z)$ are periodic of period $n$.

\begin{thm}[\text{\cite[Theorem 11.6]{CE56}}] \label{thm:periodic-coh}
Let $G$ be a finite group. Then the following are equivalent:
\begin{enumerate}[\normalfont(i)]
\item $G$ has periodic cohomology.
\item $G$ has no subgroup of the form $C_p^2$ for $p$ prime.
\item The Sylow subgroups of $G$ are cyclic or (generalised) quaternionic $Q_{2^n}$.
\end{enumerate}
\end{thm}

Recall that a \textit{binary polyhedral group} is a finite non-cyclic subgroup of $S^3$ and consists of the quaternion groups $Q_{4n}$ for $n \ge 2$ and the binary tetrahedral, octahedral and icosahedral groups $\widetilde{T}$, $\widetilde{O}$, $\widetilde{I}$. The following is well known:

\begin{prop}[\text{\cite[Proposition 1.3]{Ni21}}] \label{prop:eichler-implies-BPG-quotient}
A finite group $G$ satisfies the Eichler condition if and only if $G$ has no quotient which is a binary polyhedral group.
\end{prop}

We now aim to prove the following, which is the main result of this section.

\begin{thm} \label{thm:main3}
Suppose $G$ has periodic cohomology and does not satisfy the Eichler condition. Then $G$ either has a quotient of the form $Q_{4n}$ for $n \ge 6$ or a binary polyhedral quotient $H$ for which $m_{\H}(G)=m_{\H}(H)$.
\end{thm}

In order to prove this, we will need:

\begin{prop}[\text{\cite[Proposition 3.3]{Ni21}}] \label{prop:relative-eichler-group}
Let $N \le G$ be a normal subgroup and let $H = G/N$. Then $m_{\H}(G)=m_{\H}(H)$ if and only if $N$ is contained in all normal subgroups $N'$ for which $G/N'$ is a binary polyhedral group.
\end{prop}

We will also require the following three lemmas.
Let $\Syl_p(G)$ denote the isomorphism class of a Sylow $p$-subgroup of $G$ for $p$ prime.

\begin{lemma} \label{lemma:quotient-dihedral-quaternionic}
The proper quotients of $D_{2^n}$ and $Q_{2^n}$ are either $C_2$ or of the form $D_{2^m}$ for $2 \le m \le n-1$.
\end{lemma}

\begin{lemma} \label{lemma:quotient-sylow}
If $1 \to N \to G \to H \to 1$ is an extension, then there is an extension of abstract groups
$ 1 \to \Syl_p(N) \to \Syl_p(G) \to \Syl_p(H) \to 1$
for every prime $p$.
\end{lemma}

By combining these two lemmas with Theorem \ref{thm:periodic-coh} (iii), we see that any quotient $H$ of a group $G$ with periodic cohomology has $\Syl_p(H)$ cyclic for $p$ odd and $\Syl_2(H)$ cyclic, dihedral or quaternionic.

\begin{lemma} \label{lemma:disjoint-quotients}
Suppose $G$	has periodic cohomology and that there exists non-trivial disjoint normal subgroups $N$, $N' \le G$ such that $H=G/N$ and $H'=G/N'$ are binary polyhedral groups. Then $G$ has a quotient of the form $Q_{4n}$ for some $n \ge 6$. Furthermore $H$ and $H'$ are not of the form $Q_{2^n}$ for any $n \ge 3$, $\widetilde{T}$, $\widetilde{O}$ or $\widetilde{I}$.
\end{lemma}

\begin{proof}
Suppose for contradiction that $H$ and $H'$ are not of the form $Q_{4n}$ for $n \ge 6$. Therefore $H$ and $H'$ are 
each of the form $Q_8$, $Q_{12}$, $Q_{16}$, $Q_{20}$, $\widetilde{T}$, $\widetilde{O}$ or $\widetilde{I}$. 

Since $N$ and $N'$ are disjoint, $N \times N' \le G$ is a normal subgroup and so $H$ and $H'$ have a common quotient
$ G/(N \times N') \cong H/N' \cong H'/N$.
Furthermore, since $G$ has periodic cohomology, Theorem \ref{thm:periodic-coh} (ii) implies that $|N|$ and $|N'|$ are coprime.
The remainder of the proof will be split into two cases. 

First suppose that $\Syl_2(G)$ is cyclic. By Lemma \ref{lemma:quotient-sylow}, $\Syl_2(H)$ and $\Syl_2(H')$	must be quotients of $\Syl_2(G)$ and so are also cyclic. This implies $H$ and $H'$ are each of the form $Q_{12}$ or $Q_{20}$. Since $|N|$ and $|N'|$ are non-trivial and coprime, $|H| \ne |H'|$ and so we can assume that $H=Q_{12}$ and $H'=Q_{20}$.	These groups have common quotients $1$, $C_2$ and $C_4$ and the restriction that $|N|$ and $|N'|$ be coprime implies that $N=C_5$ and $N'=C_3$. This implies, for example, that $G/C_3 \cong Q_{20}$. Since $(3,20)=1$, this extension must split and so $G \cong C_3 \rtimes_{\varphi} Q_{20}$ for some map $\varphi: Q_{20} \to \Aut(C_3) \cong C_2$. If $\varphi=1$, then $G \cong Q_{20} \times C_3$ which does not have $Q_{12}$ as a quotient. The only other option is that $\varphi$ is the quotient by $C_{10}$ which implies that $G \cong Q_{60}$.

Now suppose that $\Syl_2(G) = Q_{2^n}$ for some $n \ge 3$. Similarly Lemma \ref{lemma:quotient-sylow} implies that $\Syl_2(H)$ and $\Syl_2(H')$	 are quotients of $Q_{2^n}$. Since $4 \mid |H|, |H'|$ we can deduce, by Lemma \ref{lemma:quotient-dihedral-quaternionic}, that $\Syl_2(H)$, $\Syl_2(H')=Q_{2^n}$ and so $|N|$, $|N'|$ are odd. Now $H$ and $H'$ are each of the form $Q_8$, $Q_{16}$, $\widetilde{T}$, $\widetilde{O}$ or $\widetilde{I}$ and it is easy to verify that these groups have no non-trivial normal subgroups of odd order, which is a contradiction. 

For the last part note that, if $H$ or $H'$ were of the form $Q_{2^n}$, $\widetilde{T}$, $\widetilde{O}$ or $\widetilde{I}$, then we can get a contradiction using the same argument in the previous paragraph.
\end{proof}

\begin{proof}[Proof of \cref{thm:main3}]
Since $G$ does not satisfy the Eichler condition, there exists a binary polyhedral quotient $H=G/N$ which we can pick to have maximal order. Since $m_{\H}(G) \ne m_{\H}(H)$, Proposition \ref{prop:relative-eichler-group} implies that there exists a binary polyhedral quotient $H'=G/N'$ for which $N \not \subseteq N'$, and $N' \not \subseteq N$ also by maximality of $|H|$.
Now $G$ has quotient $\widehat{G}=G/(N \cap N')$ which has $H=G/K$ and $H'=G/K'$ for $K=N/(N \cap N')$ and $K'=N'/(N \cap N')$ disjoint normal subgroups, by the third isomorphism theorem. In addition, $K$ and $K'$ are non-trivial since $N \cap N' \ne N, N'$.

If $\widehat{G}$ has periodic cohomology, then Lemma \ref{lemma:disjoint-quotients} implies that $G$ has a quotient of the form $Q_{4n}$ for $n \ge 6$.	Suppose that $\widehat{G}$ does not have periodic cohomology. Since $\Syl_p(\widehat{G})$ is a quotient of $\Syl_p(G)$ for all $p$ prime by Lemma \ref{lemma:quotient-sylow}, combining Lemma \ref{lemma:quotient-dihedral-quaternionic} and Theorem \ref{thm:periodic-coh} (iii) shows that we must have $\Syl_2(\widehat{G})=D_{2^n}$ for some $n \ge 2$. Now Lemma \ref{lemma:quotient-sylow} also implies that $\Syl_2(H)$ and $\Syl_2(H')$ are quotients of $D_{2^n}$. However, by Lemma \ref{lemma:quotient-dihedral-quaternionic}, the only cyclic or generalised quaternionic quotients of $D_{2^n}$ are $1$ or $C_2$ which contradicts the fact that $4 \mid |H|, |H'|$.
\end{proof}

Recall that $m_{\H}(Q_{4n}) = \lfloor n/2 \rfloor$ by \cite[Section 12]{Jo03a} and, using tables of real representations, it can be shown that $m_{\H}(\widetilde{T})=1$, $m_{\H}(\widetilde{O})=2$ and $m_{\H}(\widetilde{I})=2$.

\begin{cor} \label{cor:quot}
	If $G$ has periodic cohomology, then the following are equivalent:
	\begin{enumerate}[\normalfont(i)]
	\item $G$ has no quotient of the form $Q_{4n}$ for $n \ge 6$.
	\item $m_{\H}(G) \le 2$.
	\item $G$ has a binary polyhedral quotient $H$ for which $m_{\H}(G)=m_{\H}(H)\le 2$.
	\end{enumerate}
\end{cor}

\begin{proof}
If $G$ has quotient $Q_{4n}$ for $n \ge 6$, then $m_{\H}(G) \ge m_{\H}(Q_{4n}) \ge 3$ by lifting the one dimensional quaternionic representations. Converesely, if $G$ has no quotient of the form $Q_{4n}$ for $n \ge 6$, then Theorem \ref{thm:main3} implies that either $m_{\H}(G)=0$ or $G$ has binary polyhedral quotient $H$ for which $m_{\H}(G) = m_{\H}(H)$. Since $H$ is not of the form $Q_{4n}$ for $n \ge 6$, the results stated above imply that $m_{\H}(H) \le 2$.	
\end{proof}

We have the following improvement to \cref{thm:main3} in the case where $G$ has a quotient of the form $\wt T$, $\wt O$ or $\wt I$.

\begin{prop} \label{prop:exceptional-quotients}
If $G$ has periodic cohomology and has quotient $H = \widetilde{T}$, $\widetilde{O}$ or $\widetilde{I}$, then $m_{\H}(G)=m_{\H}(H) \le 2$. In particular, $G$ has no quotient of the form $Q_{4n}$ for $n \ge 6$.
\end{prop}

\begin{proof}
If $m_{\H}(G) \ne m_{\H}(H)$, then Proposition \ref{prop:eichler-implies-BPG-quotient} implies that there exists a binary polyhedral quotient $H'=G/N'$ for which $N \not \subseteq N'$. Since $\widetilde{T}$, $\widetilde{O}$ and $\widetilde{I}$ have no proper quotients which are binary polyhedral groups, we must also have $N' \not \subseteq N$ and so the group $\widehat{G} = G/(N \cap N')$ satisfies the conditions of Lemma \ref{lemma:disjoint-quotients} provided $\widehat{G}$ has periodic cohomology.
Note that $\Syl_2(\widehat{G})$ has quotient $\Syl_2(H) = Q_8$ or $Q_{16}$ by Lemma \ref{lemma:quotient-sylow} and so $\Syl_2(\widehat{G})$ is not dihedral by Lemma \ref{lemma:quotient-dihedral-quaternionic}. 
This implies that $\widehat{G}$ has periodic cohomology and Lemma \ref{lemma:disjoint-quotients} then contradicts the fact that $H$ is of the form $\widetilde{T}$, $\widetilde{O}$ or $\widetilde{I}$. The second part now follows by Corollary \ref{cor:quot}.
\end{proof}

We now recall the classification of groups with $4$-periodic cohomology using the notation of Milnor \cite{Mi57}. This can also be found in \cite[Chapter 7]{Jo03a}.

Throughout, we will write $C_n \rtimes_{(r)} C_m$ to denote the semi-direct product where the generator $x \in C_m$ acts on the generator $y \in C_n$ by $xyx^{-1}=y^r$ for some $r \in \Z$. We also assume each family contains $G \times C_n$ for any $G$ listed with $(n, |G|)=1$.

\begin{enumerate}[\normalfont(i)$'$]
\item $C_n$ for $n \ge 1$, the cyclic group of order $n$.
\item $D_{4n+2}$ for $n \ge 1$, the dihedral group of order $4n+2$.
\item $Q_{4n} = \langle x,y \mid x^n=y^2, yxy^{-1}=x^{-1} \rangle$ for $n \ge 2$ and $\widetilde{T},\widetilde{O},\widetilde{I}$.
\item $D(2^n,m)=C_m \rtimes_{(-1)} C_{2^n}$ for $n \ge 3$ and $m \ge 3$ odd.
\item $P_{8 \cdot 3^n}' = Q_8 \rtimes_{\varphi} C_{3^n}$ for $n \ge 2$, where $\varphi: C_{3^n} \to \Aut Q_8$ sends the generator $z \in C_{3^n}$ to $\varphi(z): x \mapsto y, y \mapsto xy$.
\item $P_{48 n}'' = C_n \cdot \widetilde{O}$ for $n \ge 3$ odd, the not-necessary-split extension which has cyclic Sylow $3$-subgroup and has action $\widetilde{O} \twoheadrightarrow \widetilde{O}/\widetilde{T} = C_2 \le \Aut C_n$.
\item $Q(2^na;b,c) = (C_a \times C_b \times C_c) \rtimes_{\varphi} Q_{2^n}$  for $n \ge 3$ and $a,b,c \ge 1$ odd coprime with $b > c$. If $C_a = \langle p \rangle$, $C_b = \langle q \rangle$ and $C_c = \langle r \rangle$, then the action is given by $ \varphi(x) : p \mapsto p^{-1}, q \mapsto q^{-1}, r \mapsto r \quad \varphi(y) : p \mapsto p^{-1}, q \mapsto q, r \mapsto r^{-1}$.
\end{enumerate}

We will now prove the following, which was stated in the introduction.

\begin{thm} \label{thm:4-periodic-list}
The groups $G$ with $4$-periodic cohomology for which $m_{\H}(G) \le 2$ are as follows, where each family contains $G \times C_n$ for any $G$ listed with $(n, |G|)=1$.
\begin{enumerate}[\normalfont(i)]
\item $C_n$ for $n \ge 1$.
\item $D_{4n+2}$ for $n \ge 1$.
\item $Q_8, Q_{12},Q_{16},Q_{20},\widetilde{T},\widetilde{O},\widetilde{I}$.
\item $D(2^n,3)$, $D(2^n,5)$ for $n \ge 3$.
\item $P_{8 \cdot 3^n}'$ for $n \ge 2$.
\item $P_{48 n}''$ for $n \ge 3$ odd.
\item $Q(16;m,n)$ for $m > n \ge 1$ odd coprime.
\end{enumerate}
\end{thm}

\begin{proof}
First note that we can ignore the groups of the form $G \times C_n$ for $G$ listed and $(n,|G|)=1$ since $m_{\H}(G \times C_n) = m_{\H}(G)$ in these cases.

It can be shown that the groups in (i)$'$, (ii)$'$ satisfy the Eichler condition \cite[Section 12]{Jo03a}. For the groups $G$ in (iii)$'$, we use that $m_{\H}(Q_{4n}) = \lfloor n/2 \rfloor$, $m_{\H}(\widetilde{T})=1$, $m_{\H}(\widetilde{O})=2$ and $m_{\H}(\widetilde{I})=2$ as mentioned previously.

In case (iv)$'$, suppose $G$ has a binary polyhedral quotient $H$. Explicit computation shows that $Z(H)=C_2$ and so the quotient map $f: G \twoheadrightarrow H$ must have $f(Z(G)) \subseteq Z(H) = C_2$. If $x \in C_m$ and $y \in C_{2^n}$ are generators, it is easy to see that $Z(D(2^n,m)) = \langle y^2 \rangle = C_{2^{n-1}}$ which has index two subgroup $N = \langle y^4 \rangle$. Hence $f$ factors through $G/\langle y^4 \rangle = C_m \rtimes_{(-1)} C_4 = Q_{4m}$.
By Proposition \ref{prop:relative-eichler-group}, we have that $m_{\H}(G)=m_{\H}(Q_{4m}) = (m-1)/2$ since $m$ is odd and so $m_{\H}(G) \le 2$ if and only if $m = 3$ or $5$ and any $n \ge 3$.

The groups in (v)$'$ all have quotient $\widetilde{T}$ and so $m_{\H}(P'_{8 \cdot 3^n})=1$ by \cref{prop:exceptional-quotients}. Similarly the groups in (vi)$'$ have quotient $\widetilde{O}$ and so $m_{\H}(P''_{48n})=2$. 
For the groups in (vii)$'$, suppose $G=Q(2^ka;b,c)$ has $m_{\H}(G) \le 2$ for $a, b, c \ge 1$ odd coprime with $b >c$. Recall that, by Corollary \ref{cor:quot}, $m_{\H}(G) \le 2$ if and only if $G$ has no quotient of the form $Q_{4n}$ for $n \ge 6$. Since $G$ has quotient $Q_{2^ka}$ for $k \ge 3$, this implies that $a=1$ and $k=3$ or $4$. 
If $k = 3$, then it is easy to see that $G \cong Q(8c;1,b) \cong Q(8b;c,1)$ and so has quotients $Q_{8b}$ and $Q_{8c}$. Hence $b=c=1$ which contradicts  $b>c$. 

Now suppose $k=4$, which we write as $G=Q(16;m,n) = (C_n \times C_m) \rtimes Q_{16}$ for $m > n \ge 1$ odd coprime. If $N' = C_n \times C_m$, then $Q_{16} = G/N'$.
If $m_{\H}(G) \ne 2$, then Proposition \ref{prop:relative-eichler-group} implies that $G$ has another binary polyhedral quotient $H=G/N$ such that $N' \not \subseteq N$.
If $\widehat{G} = G/(N \cap N')$, then $\Syl_2(\widehat{G})$ has quotient $Q_{16}$ by \ref{lemma:quotient-sylow} which implies that $\Syl_2(\widehat{G})$ is not dihedral and so $\widehat{G}$ has periodic cohomology. If $N \not \subseteq N'$, we could then apply Lemma \ref{lemma:disjoint-quotients} to get a contradiction since $Q_{16}=G/N'$. 

Hence we can assume that $N \subseteq N' = C_n \times C_m = C_{nm}$ and so $N$ is of the form $C_{n'} \times C_{m'}$ for $n' \le n$ and $m' \le m$. It is easy to see that
\[ H = G/N = (C_a \times C_b) \rtimes_{\varphi'} Q_{16} = Q(16;a,b) \]
where $a=n/n'$ and $b=n/n'$. It is then a straightforward exercise to check that this is not a binary polyhedral group unless $a=b=1$. This implies that $N=N'$ which contradicts the fact that $N' \not \subseteq N$.
\end{proof}

\begin{remark}
For the groups in the list, the groups in (i)$'$, (ii)$'$ have $m_{\H}(G)=0$, and the groups $Q_8 \times C_n$, $Q_{12} \times C_n$ and $\widetilde{T} \times C_n$ from (iii)$'$ and the groups in (v)$'$ all have $m_{\H}(G)=1$. All other groups in the list have $m_{\H}(G)=2$.
\end{remark}

We conclude this section by noting the following which we will use in Section \ref{section:cancellation-periodic-groups}.

\begin{cor} \label{cor:mh>3}
Suppose $G$ has $4$-periodic cohomology and $m_{\H}(G) \ge 3$. Then either $G$ has a quotient $Q_{4n}$ for some $n \ge 7$ or $G \cong Q_{24} \times C_n$ where  $(n,24)=1$.
\end{cor}

\begin{proof}
If $m_{\H}(G) \ge 3$ and $G$ does not have a quotient $Q_{4n}$ for $n \ge 7$, then  Theorem \ref{thm:main3} implies that $G$ has a quotient $Q_{24}$. This rules out the groups in (i)$'$, (ii)$'$, (iv)$'$, (v)$'$ and (vi)$'$ by the proof of Theorem \ref{thm:4-periodic-list}. If $G$ is in (vii)$'$, then $G=Q(2^ka;b,c) \times C_n$ for $k \ge 3$ and $a, b, c ,n \ge 1$ odd coprime with $b > c$. Since $G$ has no quotient $Q_{4n}$ for $n \ge 7$, we must have that $k=3$. Since $Q(8a;b,c) \cong Q(8b;c,a) \cong Q(8c;a,b)$, $G$ has quotients $Q_{8a}$, $Q_{8b}$ and $Q_{8c}$ and so $a, b, c \le 3$ which is a contradiction. Hence $G$ is in (iii)$'$ which implies that $G \cong Q_{24} \times C_n$ for some $n \ge 1$ with $(n,24)=1$.
\end{proof}

\section{Cancellation over groups with periodic cohomology} \label{section:cancellation-periodic-groups}

The aim of this section will be to prove the following two cancellation theorems, the latter of which will complete the proof of \cref{thm:D2}.

\begingroup
\begin{thm} \label{thm:main4}
Let $G$ have periodic cohomology. 
Then $\Z G$ has {\normalfont SFC} if and only if $m_{\H}(G) \le 2$.
\end{thm}
\endgroup

\begingroup
\begin{thm} \label{thm:main5}
Let $G$ have $4$-periodic cohomology and let $P_G$ be a projective $\Z G$ module for which $\sigma_4(G) = [P_G] \in C(\Z G)/T_G$. 
Then $[P_G]$ has cancellation if and only if $m_{\H}(G) \le 2$.
\end{thm}
\endgroup

We will begin by proving \cref{thm:main4}. Firstly note that, if $G$ satisfies the Eichler condition, then \cref{thm:jacobinski} implies that $\Z G$ has SFC. In the case where $G$ does not satisfy the Eichler condition, we have the following result of the author:

\begin{thm}[\text{\cite[Theorem B]{Ni21}}] \label{thm:jkn-B}
Suppose $G$ has a binary polyhedral quotient $H$ such that $m_{\H}(G)=m_{\H}(H)$. Then $\Z G$ has {\normalfont SFC} if and only if $\Z H$ has {\normalfont SFC}.
\end{thm}

This is a consequence of \cite[Theorem A]{Ni21}, which itself can be obtained by specialising \cref{thm:main2} to $\bar{P}=\Z G$ and combining with the following two results.

\begin{thm}[\text{\cite[Theorems 7.15-7.18]{MOV83}}] \label{thm:MOV}
If $H = Q_8, Q_{12}, Q_{16}, Q_{20}, \widetilde{T}, \widetilde{O}, \widetilde{I}$, then the map $\Z H^\times \to K_1(\Z H)$ is surjective.	
\end{thm}

\begin{thm}[\text{\cite[Theorem I]{Sw83}}] \label{thm:swan-BPG}
If $G$ is a binary polyhedral group, then $\Z G$ has {\normalfont SFC} if and only if $G$ is of the form $Q_8, Q_{12}, Q_{16}, Q_{20}, \widetilde{T}, \widetilde{O}, \widetilde{I}$.
\end{thm}

The latter is a special case of \cref{thm:main4} since these are the binary polyhedral groups for which $m_{\H}(G) \le 2$.

\begin{proof}[Proof of \cref{thm:main4}]
Suppose $\Z G$ has SFC. By combining Theorems \ref{thm:cancellation-closed-under-quotients} and \ref{thm:swan-BPG}, we get that $G$ has no quotient of the form $Q_{4n}$ for $n \ge 6$. By \cref{cor:quot}, this implies that $m_{\H}(G) \le 2$.

Suppose instead that $m_{\H}(G) \le 2$. By \cref{cor:quot}, we get that $G$ has a binary polyhedral quotient $H$ for which $m_{\H}(G)=m_{\H}(H) \le 2$. By \cref{thm:swan-BPG}, this implies that $\Z H$ has SFC and so $\Z G$ has SFC by \cref{thm:jkn-B}.
\end{proof}

We now turn to the proof of Theorem \ref{thm:main5}. We require the following results:

\begin{lemma}[\text{\cite[Theorem I]{Sw83}}] \label{thm:total-non-cancellation}
If $G = Q_{4n}$ for $n \ge 7$, then $[P]$ has non-cancellation for every projective $\Z G$-module $P$.
\end{lemma}

\begin{lemma}[\text{\cite[Proposition 53.10]{CR90}}] \label{lemma:swan-subgroup-lifts} Let $N \le G$ be a normal subgroup. Then the induced map $- \otimes_{\Z N} \Z : T_G \to T_{G/N}$ is surjective.
\end{lemma}

The main tool in the proof of \cref{thm:main5} will be the following lemma. This is the only place in this article that we will use \cref{thm:main2} instead of \cref{thm:jkn-B}, which is strictly weaker and was already proven in \cite{Ni21}.

\begin{lemma} \label{lemma:cancellation}
Let $G$ have $4$-periodic cohomology and let $P_G$ be a projective $\Z G$ module for which $\sigma_4(G) = [P_G] \in C(\Z G)/T_G$. Suppose there exists a normal subgroup $N \le G$ for which $H = G/N$ is a binary polyhedral group, $m_{\H}(G)=m_{\H}(H) \le 2$ and $[P_G \otimes_{\Z N} \Z] \in T_H$. Then $[P_G]$ has cancellation.	
\end{lemma}

\begin{proof}
It follows from \cref{thm:main1} that the property that $[P_G]$ has cancellation is independent of the choice of representative $P_G$ for $\sigma_4(G)$. Therefore, it suffices to determine cancellation for any choice of $P_G$.

By Lemma \ref{lemma:swan-subgroup-lifts}, there exists $[P] \in T_G$ with $[P \otimes_{\Z N} \Z] = -[P_G \otimes_{\Z N} \Z] \in C(\Z H)$. By \cite[Theorem A]{Sw60a}, there exists pa rojective $\Z G$ module $P_G'$ of rank one such that $P_G \oplus P \cong P_G' \oplus \Z G^r$. This satisfies $[P_G'] = [P_G] \in C(\Z G)/T_G$ and $[P_G' \otimes_{\Z N} \Z ]= 0 \in C(\Z H)$. 
By \cref{thm:main4}, $\Z H$ has SFC and so $P_G' \otimes_{\Z N} \Z = \Z H$.

By \cref{thm:MOV}, we have that $\Aut(\Z H) = \Z H^\times \twoheadrightarrow K_1(\Z H)$. Hence the conditions of \cref{thm:main2} are satisfies and so $[P_G']$ has cancellation.
\end{proof}

It follows from results of J. A. Wolf \cite[Chapter 6]{Wo74} that, if $G$ is in (i)$'$, (iii)$'$, (iv)$'$, (v)$'$, then $G$ is a fixed-point free finite subgroups of $SO(4)$ and so is the fundamental group of a closed 3-manifold $M$. Since $C_*(\wt M) \in \Proj^4_{\Z G}(\Z,\Z;0)$, this implies that $\sigma_4(G)=0$ for these groups. This also holds for the groups in (ii)$'$ (see \cite[p236]{Jo03a} or \cite[Section 12]{Wa79b}).

\begin{lemma} \label{thm:finiteness-obstruction-computation}
If $G$ is in $\text{\normalfont (i)}'$-$\text{\normalfont (v)}'$, then $\sigma_4(G)=0 \in C(\Z G)/ T_G$. 
\end{lemma}

The following lemma is standard.
 
\begin{lemma} \label{lemma:maps-on-class-groups}
If $N \unlhd G$ and $N \le H \le G$, then there is a commutative diagram:
\[
\begin{tikzcd}
C(\Z G) \arrow[r] \arrow[d,"\text{\normalfont Res}^G_H"] & C(\Z [G/N]) \arrow[d,"\text{\normalfont Res}^{G/N}_{H/N}"] \\
C(\Z H) \arrow[r] & C(\Z [H/N])
\end{tikzcd}
\]
where the horizontal maps are induced by $- \otimes_{\Z N} \Z$.
\end{lemma}

\begin{proof}[Proof of \cref{thm:main5}]
As in the proof of \cref{lemma:cancellation}, it will suffice to prove that $[P_G]$ has cancellation or non-cancellation where $P_G$ is any representative for $\sigma_4(G)$.

To prove the forward direction, suppose $m_{\H}(G) \ge 3$. By \cref{cor:mh>3}, we know that $G$ either has a quotient of the form $Q_{4n}$ for $n \ge 7$ or $G \cong Q_{24} \times C_n$ where $(n,24)=1$.
If $G$ has a quotient $Q_{4n}=G/N$ for $n \ge 7$ and $P = P_G \otimes_{\Z N} \Z$, then $[P]$ has non-cancellation by Theorem \ref{thm:total-non-cancellation} and so $[P_G]$ has non-cancellation also by Theorem \ref{thm:cancellation-closed-under-quotients}. 
If $G = Q_{24} \times C_n$, then $\sigma_4(G)=0$ by \cref{thm:finiteness-obstruction-computation} and so we can take $[P_G] = [\Z G]$. Note that $[\Z Q_{24}]$ non-cancellation by Theorem \ref{thm:swan-BPG} and so $[\Z G]$ has non-cancellation by \cref{thm:cancellation-closed-under-quotients}.

For the converse, suppose $m_{\H}(G) \le 2$. If $\sigma_4(G)=0$, then we can take $[P_G] = [\Z G]$ which has cancellation by Theorem \ref{thm:main4}.
If $\sigma_4(G) \ne 0$, then Lemma \ref{thm:finiteness-obstruction-computation} implies that $G$ is one of the groups in (vi) or (vii). 
If $G$ is in (vii), then $G = Q(16;m,n) \times C_r$ for $m, n$ odd coprime and $(r,16mn)=1$. By Theorem \ref{thm:4-periodic-list}, the quotient $Q_{16}=G/N$ has $m_{\H}(G)=m_{\H}(Q_{16})=2$.
It follows from \cite[Theorems III, VI]{Sw83} that $C(\Z Q_{16}) = T_{Q_{16}}$ and so $[P_G \otimes_{\Z N} \Z] \in T_{Q_{16}}$ for all projective $\Z G$ modules $P_G$ for which $\sigma_4(G) = [P_G] \in C(\Z G)/T_G$. Hence $[P_G]$ has cancellation by Lemma \ref{lemma:cancellation}.

If $G$ is in (vi), then $G = P_{48n}'' \times C_m$ for $n \ge 3$ odd and $m \ge 1$ with $(48n,m)=1$.
Theorem \ref{thm:4-periodic-list} implies that the quotient $\widetilde{O} = G/N$ has $m_{\H}(G) = m_{\H}(\widetilde{O})=2$ where $N = C_n \times C_m$. 
It is proven in \cite[Theorem 14.1]{Sw83} that the restriction map
$\text{\normalfont Res} : C(\Z \widetilde{O}) \to C(\Z Q_{16}) \oplus C(\Z Q_{12})$ is bijective. This restricts to an injection $\Res \mid_{T_{\widetilde{O}}}: T_{\widetilde{O}} \to T_{Q_{16}} \oplus T_{Q_{12}}$ which is necessarily bijective since $T_{\widetilde{O}} \cong \Z/2$, $T_{Q_{16}} \cong \Z/2$ and $T_{Q_{12}} = 0$ \cite[Theorem VI]{Sw83}.
Since $C(\Z Q_{16}) = T_{Q_{16}}$, a class $[P] \in C(\Z \widetilde{O})$ is contained in $T_{\widetilde{O}}$ if and only if $\Res^{\widetilde{O}}_{Q_{12}}(P) = 0 \in C(\Z Q_{12})$.

Note that $G$ has a (hyperelementary) subgroup $H=Q_{12n} \times C_m$ which contains $N$ and has $Q_{12}=H/N$. By Lemma \ref{lemma:maps-on-class-groups}, we now have a commutative diagram:
\[
\begin{tikzcd}
C(\Z G) \arrow[r] \arrow[d,"\text{\normalfont Res}^G_H"] & C(\Z \widetilde{O}) \arrow[d,"\text{\normalfont Res}^{\widetilde{O}}_{Q_{12}}"] \\
C(\Z H) \arrow[r] & C(\Z Q_{12}) 
\end{tikzcd}
\]
Let $P_G$ be such that $\sigma_4(G) = [P_G] \in C(\Z G)/T_G$. If $P_H = \Res^G_H(P_G)$, then it is straightforward to see that $\sigma_4(H) = [P_H] \in C(\Z H)/T_H$.
Since $H$ is in (iii), \cref{thm:finiteness-obstruction-computation} implies that $\sigma_4(H)=0$ and so $[P_H] \in T_H$. This implies that $[P_H \otimes_{\Z N} \Z] =0$ since it is contained in $T_{Q_{12}}=0$. 
Hence $\Res^G_H(P_G) \otimes_{\Z N} \Z = 0 \in C(\Z Q_{12})$ and so, as the diagram above commutes, we have $\Res^{\widetilde{O}}_{Q_{12}}(P_G \otimes_{\Z N} \Z) = 0 \in C(\Z Q_{12})$. Hence $[P_G \otimes_{\Z N} \Z] \in T_{\widetilde{O}}$ which implies that $[P_G]$ has cancellation by Lemma \ref{lemma:cancellation}.
\end{proof}

\begin{proof}[Proof of \cref{thm:D2}]
Since $\D(G)$ is connected, finite D2 complexes $X$ with $\pi_1(X) = G$ are determined up to homotopy by $\chi(X)$ if and only if $\D(G)$ has cancellation. 

Let $P_G$ be such that $\sigma_4(G) = [P_G] \in C(\Z G)/T_G$. By combining Theorems \ref{thm:main0} and \ref{thm:main1}, we get that there is an isomorphism of graded trees
\[ \Psi \circ \wt C_* : \D(G) \to [P_G] \]
and so $\D(G)$ has cancellation if and only if $[P_G]$ has cancellation. By \cref{thm:main5}, $[P_G]$ has cancellation if and only if $m_{\H}(G) \le 2$, which completes the proof.
\end{proof}

\section{Balanced presentations and the D2 problem} \label{section:balanced-presentations}

We will now show how the following can be deduced from \cref{thm:D2}.

\begingroup
\renewcommand\thethm{\ref{thm:D2-2}}
\begin{cor} 
Let $G$ have $4$-periodic cohomology. Then:
\begin{enumerate}[\normalfont(i)]
\item If $G$ has the {\normalfont D2} property, then $G$ has a balanced presentation.
\item If $G$ has a balanced presentation and $m_{\H}(G) \le 2$, then $G$ has the {\normalfont D2} property.
\end{enumerate}
\end{cor}
\endgroup

\begin{proof}
Recall that, as discussed at the end of \cref{section:PHT}, $G$ has the D2 property if and only if the induced map
	$\Psi \circ \wt C_* : PHT(G,2) \to [P_G]$
	is bijective. 
	
Suppose $G$ has the D2 property. Then $\Psi \circ \wt C_*$ is bijective and, since $[P_G]$ contains a projective $\Z G$ module of rank one, there must exist $(X,\rho) \in PHT(G,2)$ such that $\chi(X) = 1$. If $\mathcal{P}$ is a presentation such that $X \simeq X_{\mathcal{P}}$, then $\Def(\mathcal{P}) = 1 - \chi(X) = 0$ and so $\mathcal{P}$ is a balanced presentation as required.
	
Suppose instead that $m_{\H}(G) \le 2$ and that $G$ has a balanced presentation $\mathcal{P}$. By \cref{thm:D2}, $[P_G]$ has cancellation and so $[P_G] = \{ P_0 \oplus \Z G^r : r \ge 0\}$ for some projective $P_0$ of rank one. Let $P_1 = \Psi(\wt C_*(X_{\mathcal{P}}))$, which has $\rank(P_1)=1$ since $\chi(X_{\mathcal{P}})=1$. Since $P_1 \in [P]$, this implies that $P_1 \cong P_0$. In particular, for all $r \ge 0$, we have $\Psi(\wt C_*(X_{\mathcal{P}} \vee r S^2)) \cong P_0 \oplus \Z G^r$ and so $\Psi \circ \wt C_*$ is surjective. Since $\Psi \circ \wt C_*$ is injective, this implies that it is bijective and so $G$ has the D2 property.
\end{proof}

\begingroup
\renewcommand\thethm{\ref{thm:poincare}}
\begin{cor}
Let $X$ be a finite Poincar\'{e} $3$-complex with $G=\pi_1(X)$ finite. Then:
\begin{enumerate}[\normalfont(i)]
\item If $X$ has a cell structure with a single $3$-cell, \hspace{-0.5mm}then $G$ \hspace{-0.5mm}has a balanced presentation.
\item If $G$ has a balanced presentation and $m_{\H}(G) \le 2$, then $X$ has a cell structure with a single $3$-cell.
\end{enumerate}
\end{cor}
\endgroup
\setcounter{thm}{0} 

\begin{proof}
Let $X$ be a finite Poincar\'{e} 3-complex with $G=\pi_1(X)$ finite. By \cref{cor:pi_1-of-poinc}, $G$ has 4-periodic cohomology.

Suppose $G$ has a cell structure with a single 3-cell $e^3$. Then $X = X^{(2)} \cup e^3$, where the 2-skeleton $X^{(2)}$ is a finite 2-complex. It follows from Poincar\'{e} duality that $\chi(X) = 0$ and so $\chi(X^{(2)})=1$. If $\mathcal{P}$ is a presentation such that $X^{(2)} \simeq X_{\mathcal{P}}$, then $\Def(\mathcal{P}) = 1 - \chi(X^{(2)}) = 0$ and so $\mathcal{P}$ is a balanced presentation as required.

Suppose instead that $m_{\H}(G) \le 2$ and that $G$ has a balanced presentation. By \cref{thm:D2-2}, $G$ has the D2 property. By \cite[Theorem 2.4]{Wa67}, there exists a finite D2 complex $K$, a map $\alpha:S^2 \to K$ and a homotopy equivalence $X  \to  K \cup_{\alpha} e^3$. 
Since $G$ has the D2 property, there is a homotopy equivalence $h : K \to K_0$ for some finite 2-complex $K_0$. By combining these maps, we obtain a homotopy equivalence $X \to K_0 \cup_{h \circ \alpha} e^3$ which gives a cell structure for $X$ with a single 3-cell as required.
\end{proof}

For the remainder of this section, we will discuss the implications of Corollaries \ref{thm:D2-2} and \ref{thm:poincare} on the D2 problem over groups with $4$-periodic cohomology.

\subsection{Balanced presentations for groups with periodic cohomology}

If $G$ has periodic cohomology, then $H_2(G;\Z)=0$ (see, for example, \cite{Sw71}). In particular, by \cref{thm:D2-2}, the groups $G$ with $4$-periodic cohomology are either counterexamples to the D2 problem or give a new supply of groups with efficient presentations. 

This gives some response to comments made by L. G. Kov\'{a}cs \cite[p212]{Ko94} and J. Harlander \cite[p167]{Har98} on the scarcity of efficient finite groups. In contrast, there is currently no example known to contradict the following conjecture:

\begin{conj}
If $G$ has periodic cohomology, then $G$ has a balanced presentation.	
\end{conj}

By \cite[Chapter 6]{Wo74}, the groups in (i)$'$, (iii)$'$, (iv)$'$, (v)$'$ are all finite 3-manifold groups which are well-known to have balanced presentations. This follows, for example, from \cref{thm:poincare} since 3-manifolds have cell structures with a single 3-cell. The groups in (ii)$'$ also have balanced presentations (see \cite[p236]{Jo03a}).

\begin{prop} \label{prop:3-manifold-group}
If $G$ is in {\normalfont (i)$'$-(v)$'$}, then $G$ has a balanced presentation.	
\end{prop}

We have not been able to find balanced presentations for any of the groups in (vi), but have succeeded for the following groups in (vii). These groups overlap in the case where $k=n=1$ with the groups in \cite[Theorem 3.1]{Ne89}.

\begin{prop} \label{prop:D2-prop}
If $n \ge 3$ and $a, b, k \ge 1$ odd coprime, then: 
\[Q(2^na;b,1) \times C_k \cong \langle x,y \mid y^kx^b y^k = x^b, xyx=y^{2^{n-2}a-1} \rangle.\]\end{prop}

By combining Propositions \ref{prop:3-manifold-group} and \ref{prop:D2-prop} with \cref{thm:D2-2}, we obtain:

\begin{thm} \label{thm:D2-property}
Suppose $G$ is in $\text{\normalfont(i)}$-$\text{\normalfont(v)}$ or has the form $Q(16;n,1) \times C_k$ for some $n, k \ge 1$ odd coprime. Then $G$ has the {\normalfont D2} property.
\end{thm}

Note that not all of these groups have free period $4$; an example is $Q(16;3,1)$ \cite{Da82}.
The simplest groups that we have not been able to find balanced presentations for are $P_{48 \cdot 3}''$ and $Q(16;3,5)$. The following is therefore of practical interest: 

\begin{question}
Do $P_{48 \cdot 3}''$ or $Q(16;3,5)$ have balanced presentations?	
\end{question}

These correspond to the groups $G_{144}^{31}$, $G_{240}^{22}$ in GAP's the Small Groups Library.

\subsection{Potential counterexamples to the D2 problem}  Recall that a $\Z G$ module $M$ has the \textit{Swan property} (or is a `Swan module') if 
\[ d_{\Z G}(M) = \max_{p \mid |G|} d_{\Z_p G}(M \otimes \Z_p)\] 
where $d_R(\,\cdot\,)$ is the number of $R$-module generators.

In 1977, J. M. Cohen proposed the group $Q_{32}$ as a counterexample to the D2 problem \cite[Section 4]{Co77}.
More generally, he conjectured the following.

\begin{conj}[\text{\cite[Problem D3]{Wa79a}}] \label{conj:cohen}
Let $X$ be a finite {\normalfont D2} complex. Then $X$ is homotopy equivalent to a finite $2$-complex if and only if $\pi_2(X)$ has the Swan property or $\pi_1(X)$ is infinite.
\end{conj}

The following was proved by Cohen as a consequence of \cite[Theorem 3]{Co78}, where $I$ is the augmentation ideal.
This gives another reason why the D2 property for groups with 4-periodic cohomology is of particular interest.

\begin{prop}[\text{\cite[p415]{Co77}}] \label{prop:cohen}
Let $X$ be a finite {\normalfont D2} complex such that $\pi_2(X)$ does not have the Swan property and $G = \pi_1(X)$ is finite. Then:
\begin{enumerate}[\normalfont (i)]
\item $G$ has free period $4$.
\item $\chi(X) = 1$.
\item $\pi_2(X)$ is non-cyclic, i.e. $\pi_2(X) \not \cong I^*$.
\end{enumerate}
\end{prop}

We will now show the following as an application of \cref{thm:D2}. Recall that $\Z G$ modules $A$ and $B$ are \textit{$\Aut(G)$-isomorphic} if there exists a bijection $\varphi : A \to B$ such that, for some $\theta \in \Aut(G)$, we have $\varphi(g \cdot x) = \theta(x) \cdot \varphi(x)$ for all $x \in A$. 

\begin{cor}
Suppose the ``if" part of \cref{conj:cohen} holds. If $G$ does not have the {\normalfont D2} property, then $G$ has free period $4$ and $m_{\H}(G) \ge 3$.	
\end{cor}

\begin{proof}
 By \cref{prop:cohen}, $G$ has free period $4$ and there exists a finite D2 complex $X$ with $\pi_1(X) = G$, $\chi(X) = 1$ and $\pi_2(X) \not \cong I^*$.
Recall that, in \cref{thm:mu2=1}, we constructed a finite D2 complex $X_G$ with $\pi_1(X_G) = G$ and $\chi(X_G) = 1$ whenever $G$ has 4-periodic cohomology.
Since $G$ has free period $4$, we can take $J = I^*$ in the proof of \cref{thm:mu2=1} which gives that $\pi_2(X_G) \cong I^*$. In particular, $X \vee rS^2 \simeq X_G \vee rS^2$ for some $r \ge 1$. If $X \simeq X_G$, then $\pi_2(X)$ and $\pi_2(X_G)$ would be $\Aut(G)$-isomorphic which is a contradiction since $\pi_2(X)$ is non-cyclic and $\pi_2(X_G) \cong I^* \cong \Z G/ \Sigma$. Hence $\D(G)$ has non-cancellation and so $m_{\H}(G) \ge 3$ by \cref{thm:D2}. 
\end{proof}

We say that two presentations $\mathcal{P}$ and $\mathcal{Q}$ for a group $G$ are \textit{exotic} if $X_{\mathcal{P}} \not \simeq X_{\mathcal{Q}}$ and $\Def(\mathcal{P}) = \Def(\mathcal{Q})$ or, equivalently, if $X_{\mathcal{P}} \vee rS^2 \simeq X_{\mathcal{Q}} \vee rS^2$ for some $r \ge 0$.
The following was shown recently by Mannan and Popiel \cite{MP21}. This is, to date, the only known example of an exotic presentation for a finite non-abelian group.

\begin{thm}[\text{\cite[Theorem A]{MP21}}] \label{thm:mannan-popiel}
The quaternion group $Q_{28}$ has presentations
\[ \PP_1= \langle x,y \mid x^7=y^2, xyx=y \rangle, \quad \PP_2= \langle x,\mid x^7=y^2, y^{-1}xyx^2=x^3y^{-1}x^2y \rangle\]
such that $\pi_2(X_{\PP_1}) \not \cong \pi_2(X_{\PP_2})$ are not $\Aut(Q_{28})$-isomorphic. Hence $X_{\PP_1} \not \simeq X_{\PP_2}$.
\end{thm}

We now proceed to point out the following two consequences.

\begin{thm} \label{thm:cohen=false}
The ``only if" part of \cref{conj:cohen} is false.
\end{thm}

\begin{proof}
It is noted after the proof of \cite[Theorem A]{MP21} that $d_{\Z Q_{28}}(\pi_2(X_{\PP_1})) = 1$ and $d_{\Z Q_{28}}(\pi_2(X_{\PP_1})) \ne 1$. However, $\pi_2(X_{\PP_1}) \oplus \Z G^r \cong \pi_2(X_{\PP_2}) \oplus \Z G^r$ for some $r \ge 0$ and it follows that $\pi_2(X_{\PP_1}) \otimes \Z_p \cong \pi_2(X_{\PP_2}) \otimes \Z_p$ for $p \mid |G|$ prime since $\Z_p G$ is semisimple by Maschke's theorem.
This implies that $d_{\Z_p G}(\pi_2(X_{\PP_2}) \otimes \Z_p)=1$ for all $p \mid |G|$ and so $\pi_2(X_{\PP_2})$ does not have the Swan property. This contradicts the ``only if" direction of \cref{conj:cohen} since $X_{\mathcal{P}_2}$ is a finite 2-complex.	
\end{proof}

\begin{thm} \label{thm:q28}
$Q_{28}$ has the {\normalfont D2} property and $m_{\H}(Q_{28})=3$.
\end{thm}

In \cite{BW05}, this is proposed as a counterexample by F. R. Beyl and N. Waller and so this answers their question in the negative (see also \cite[p23]{MR18}).

\begin{proof}
Note that $Q_{28}$ is in (iii) and so has $\sigma_4(Q_{28})=0$ by \cref{thm:finiteness-obstruction-computation}. Hence, as in the proof of \cref{thm:D2-2}, it suffices to show that the injective map of graded trees $\Psi \circ \wt C_* : \PHT(Q_{28},2) \to [\Z Q_{28}]$ is in fact bijective.
By the discussion in \cref{section:D2-overview}, $\PHT(Q_{28},2)$ and $[\Z Q_{28}]$ are both forks. By \cite[Theorem III]{Sw83}, there are two rank one stably free $\Z Q_{28}$ modules and so $[\Z Q_{28}]$ has two vertices at the minimal level. By \cref{thm:mannan-popiel}, $PHT(Q_{28},2)$ has at least two vertices at the minimal level. Hence the injective map $\Psi \circ \wt C_*$ must be bijective, and so $Q_{28}$ has the D2 property.
\end{proof}		

It should be possible to replicate this proof for other examples of groups with $4$-periodic cohomology and $m_{\H}(G) \ge 3$. We expect that the quaternion groups $Q_{4n}$ contain the main difficulties associated with the case $m_{\H}(G) \ge 3$, and so we ask:

\begin{question}
Does $Q_{4n}$ have the {\normalfont D2} property for all $n \ge 2$?
\end{question}

The case $Q_{32}$ is of particular significance since it was the first proposed counterexample to the D2 problem \cite[p415]{Co77}.

\section*{Acknowledgements}
 I would like to thank my supervisor F. E. A. Johnson for many interesting conversations on the D2 problem and stable modules, W. H. Mannan for drawing my attention to his recent work on exotic presentations for $Q_{28}$ and Jonathan Hillman for many helpful comments. I am also indebted to an anonymous referee for their careful reading and many useful suggestions. In particular, they suggested I give an explicit form for the map $\Psi$ from \cref{thm:main1}, which can now be found in \cref{prop:Psi-explicit}. 
 This work was supported by the UK Engineering and Physical Sciences Research Council (EPSRC) grant EP/N509577/1.

\end{document}